\documentclass[runningheads]{llncs}
\usepackage{graphicx}
\usepackage{float}
\usepackage{mathtools}
\usepackage{amssymb, amsmath, latexsym}
\usepackage{nicefrac}
\usepackage{hhline}
\usepackage{multirow}
\usepackage{pifont}
\newcommand{\cmark}{\ding{51}}%
\newcommand{\xmark}{\ding{55}}%
\usepackage[colorinlistoftodos,bordercolor=orange,backgroundcolor=orange!20,linecolor=orange,textsize=scriptsize]{todonotes}
\usepackage{algorithm}\usepackage{algpseudocode}
\newcommand{\argmin}{\mathop{\arg\!\min}}
\newcommand{\argmax}{\mathop{\arg\!\max}}

\def \R {\mathbb R}
\def\eqdef{\overset{\text{def}}{=}}

\newcommand{\EndProof}{\begin{flushright}$\square$\end{flushright}}

\newcommand{\specialcell}[2][c]{\begin{tabular}[#1]{@{}c@{}}#2\end{tabular}}
  
\newcommand{\EE}{\mathbf{E}}

\def\R{\mathbb{R}}

\def\R{\mathbb R}

\def\EE{\mathbb E}

\def\la{\langle}
\def\ra{\rangle}

\begin{document}
\title{Gradient-Free Methods with Inexact Oracle for Convex-Concave Stochastic Saddle-Point Problem\thanks{The research of A. Beznosikov was partially supported by RFBR, project number 19-31-51001. The research of A. Gasnikov was partially supported by RFBR, project number 18-29-03071 mk and was partially supported by the Ministry of Science and Higher Education of the Russian Federation (Goszadaniye) no 075-00337-20-03.}}
\titlerunning{Gradient-Free Methods for Saddle-Point Problem}
%
\author{Aleksandr Beznosikov\inst{1,2}\and
Abdurakhmon Sadiev\inst{1}\and
Alexander Gasnikov\inst{1,2,3,4}}
\authorrunning{A. Beznosikov, A. Sadiev and A. Gasnikov}
%
\institute{Moscow Institute of Physics and Technology, Russia \and
 Sirius University of Science and Technology, Russia
\and
Institute for Information Transmission Problems RAS, Russia  \and
Caucasus Mathematical Center, Adyghe State University, Russia}
\maketitle              
\begin{abstract}
In the paper, we generalize the approach Gasnikov et. al, 2017, which allows solving (stochastic) convex optimization problems with an inexact gradient-free oracle, to the convex-concave saddle-point problem. The proposed approach works, at least, like the best existing approaches. But for a special set-up (simplex type constraints and closeness of Lipschitz constants in 1 and 2 norms) our approach reduces $\nicefrac{n}{\log n}$ times the required number of oracle calls (function calculations). Our method uses a stochastic approximation of the gradient via finite differences. In this case, the function must be specified not only on the optimization set itself but in a certain neighbourhood of it. In the second part of the paper, we analyze the case when such an assumption cannot be made, we propose a general approach on how to modernize the method to solve this problem, and also we apply this approach to particular cases of some classical sets.

\end{abstract}
\section{Introduction}\label{sec:intro}
In the last decade in the ML community, a big interest cause different applications of Generative Adversarial Networks (GANs) \cite{goodfellow2016nips}, which reduce the ML problem to the saddle-point problem, and the application of gradient-free methods for Reinforcement Learning problems \cite{sutton2018reinforcement}. Neural networks become rather popular in Reinforcement Learning \cite{langley00}. Thus, there is an interest in gradient-free methods for saddle-point problems
\begin{equation}
    \label{problem}
    \min_{x\in \mathcal{X}}\max_{y\in \mathcal{Y}}  \varphi(x, y).  
\end{equation}
One of the natural approaches for this class of problems is to construct a stochastic approximation of a gradient via finite differences. In this case, it is natural to expect that the complexity of the problem \eqref{problem} in terms of the number of function calculations is $\sim n$ times large in comparison with the complexity in terms of the number of gradient calculations, where $n = \dim \mathcal{X}+ \dim \mathcal{Y}$. Is it possible to obtain a better result? In this paper, we show that this factor can be reduced in some situation to a much smaller factor $\log n$. 

We use the technique, developed in \cite{gasnikov2017stochastic,gasnikov2016gradient} for stochastic gradient-free non-smooth convex optimization problems (gradient-free version of mirror descent \cite{nemirovski}) to propose a stochastic gradient-free version of saddle-point variant of mirror descent \cite{nemirovski} for non-smooth convex-concave saddle-point problems.

The concept of using such an oracle with finite differences is not new (see \cite{Shamir15}, \cite{duchi2013optimal}). For such an oracle, it is necessary that the function is defined in some neighbourhood of the initial set of optimization, since when we calculate the finite difference, we make some small step from the point, and this step can lead us outside the set. As far as we know, in all previous works, the authors proceed from the fact that such an assumption is fulfilled or does not mention it at all. We raise the question of what we can do when the function is defined only on the given set due to some properties of the problem.

\subsection{Our contributions}\label{ssec:contr}

In this paper, we present a new method called zeroth-order Saddle-Point Algorithm ({\tt zoSPA}) for solving a convex-concave saddle-point problem \eqref{problem}. Our algorithm uses a zeroth-order biased oracle with stochastic and bounded deterministic noise. We show that if the noise $\sim \varepsilon$ (accuracy of the solution), then the number of iterations necessary to obtain $\varepsilon-$solution on set with diameter $\Omega \subset \mathbb{R}^n$ is $\mathcal{O} \left( \frac{M^2 \Omega^2 }{\varepsilon^2}n\right)$ or $\mathcal{O} \left( \frac{ M^2 \Omega^2 }{\varepsilon^2}\log n\right)$ (depends on the optimization set, for example, for a simplex, the second option with $\log n$ holds), where $M^2$ is a bound of the second moment of the gradient together with stochastic noise (see below, \eqref{bound}).

In the second part of the paper, we analyze the 
structure of an admissible set. We give a general approach on how to work in the case when we are forbidden to go beyond the initial optimization set. Briefly, it is to consider the "reduced"{} set and work on it.

Next, we show how our algorithm works in practice for various saddle-point problems and compare it with full-gradient mirror descent.

\section{Notation and Definitions}\label{sec:notation}

We use $\la x,y \ra \eqdef \sum_{i=1}^nx_i y_i$ to define inner product of $x,y\in\R^n$ where $x_i$ is the $i$-th component of $x$ in the standard basis in $\R^n$. Hence we get the definition of $\ell_2$-norm in $\R^n$ in the following way $\|x\|_2 \eqdef \sqrt{\la x, x \ra}$. We define $\ell_p$-norms as $\|x\|_p \eqdef \left(\sum_{i=1}^n|x_i|^p\right)^{\nicefrac{1}{p}}$ for $p\in(1,\infty)$ and for $p = \infty$ we use $\|x\|_\infty \eqdef \max_{1\le i\le n}|x_i|$. The dual norm $\|\cdot\|_q$ for the norm $\|\cdot\|_p$ is defined in the following way: $\|y\|_q \eqdef \max\left\{\la x, y \ra\mid \|x\|_p \le 1\right\}$. Operator $\EE[\cdot]$ is full mathematical expectation and operator $\EE_\xi[\cdot]$ express conditional mathematical expectation. 

\begin{definition}[$M$-Lipschitz continuity]
Function $f(x)$ is $M$-Lipschitz continuous in $X\subseteq \R^n$ with $M > 0$ w.r.t. norm $\|\cdot\|$ when 
\begin{equation*}
\label{Lipschitz continuous}
| f(x)- f(y)|\leq M\|x-y\|, \quad \forall\ x, y\in X.
\end{equation*}
\end{definition}

\begin{definition}[$\mu$-strong convexity]
Function $f(x)$ is $\mu$-strongly convex w.r.t. norm $\|\cdot\|$ on $X\subseteq \R^n$ when it is continuously differentiable and there is a constant $\mu > 0$ such that the following inequality holds:
\begin{equation*}
\label{strong convexity}
f(y) \geq f(x) + \langle \nabla f(x), y-x \rangle + \frac{\mu}{2}\|y-x\|^2, \quad \forall\ x, y\in X.
\end{equation*}
\end{definition}

\begin{definition}[Prox-function]
Function $d(z): \mathcal{Z} \to \mathbb{R}$ is called prox-function if $d(z)$ is $1$-strongly convex w.r.t. $\| \cdot \|$-norm and differentiable on $\mathcal{Z}$ function. 
\end{definition}
\begin{definition}[Bregman divergence]
Let $d(z): \mathcal{Z} \to \mathbb{R}$ is prox-function. For  any  two  points  $z,w \in \mathcal{Z}$ we define Bregman divergence $V_z(w)$ associated with $d(z)$ as follows: 
\begin{equation*}
    \label{Bregman}
    V_z(w) = d(w) - d(z) - \langle \nabla d(z), w - z \rangle.
\end{equation*}
\end{definition}

We denote the Bregman-diameter $\Omega_{\mathcal{Z}}$ of $\mathcal{Z}$ w.r.t.\ $V_{z_1}(z_2)$ as \\
$\Omega_{\mathcal{Z}} \eqdef \max\{\sqrt{2V_{z_1}(z_2)}\mid z_1, z_2 \in \mathcal{Z}\}$.

\begin{definition}[Prox-operator]
Let $V_z(w)$ Bregman divergence. For all $x \in \mathcal{Z}$ define prox-operator of $\xi$:
\end{definition}
\begin{equation*}
    \text{prox}_x (\xi) = \text{arg}\min_{y \in \mathcal{Z}} \left(V_x(y) + \langle \xi , y \rangle \right).
\end{equation*}

\section{Main Result}\label{sec:main_res}

\subsection{Non-smooth saddle-point problem}

We consider the saddle-point problem \eqref{problem}, where $\varphi(\cdot, y)$ is convex function defined on compact convex set $\mathcal{X} \subset \mathbb{R}^{n_x}$,  $\varphi(x,\cdot)$ is concave function defined on compact convex set $\mathcal{Y} \subset \mathbb{R}^{n_y}$. 

We call an inexact stochastic zeroth-order oracle $\widetilde{\varphi}(x, y, \xi)$ at each iteration. Our model corresponds to the case when the oracle gives an inexact noisy function value. We have stochastic unbiased noise, depending on the random variable $\xi$ and biased deterministic noise. One can write it the following way:
\begin{eqnarray}
\label{PrSt}
    &\tilde \varphi (x,y, \xi)  = \varphi(x,y,\xi) + \delta(x,y),\\
\label{eq:PrSt1}
    &\mathbb{E}_{\xi}[ \widetilde{\varphi}(x, y, \xi)] =  \widetilde{\varphi}(x, y), \quad \mathbb{E}_{\xi}[\varphi(x,y,\xi)] = \varphi(x,y),
\end{eqnarray}
where random variable $\xi$ is responsible for unbiased stochastic noise and  $\delta(x,y)$ -- for deterministic noise.

We assume that exists such positive constant $M$ that for all $x,y\in \mathcal{X} \times  \mathcal{Y}$ we have 
\begin{equation}
\label{bound}
\|\nabla \varphi(x,y,\xi)\|_2 \leq M(\xi),\quad 
    \mathbb{E}[M^2(\xi)] = M^2.
\end{equation}
By $ \nabla {\varphi}(x,y, \xi)$ we mean a block vector consisting of two vectors $\nabla_x {\varphi}(x,y, \xi)$ and $\nabla_y {\varphi}(x,y, \xi)$.
One can prove that $\varphi(x,y,\xi)$ is $M(\xi)$-Lipschitz w.r.t. norm $\|\cdot\|_2$ and that $\|\nabla \varphi(x,y)\|_2 \leq M$.

Also the following assumptions are satisfied:
\begin{eqnarray}
\label{eq:PrSt2}
    |\widetilde{ \varphi}(x, y, \xi) - \varphi(x, y, \xi)| = |\delta (x, y)|\leq \Delta.
\end{eqnarray}

For convenience, we denote $\mathcal{Z} =  \mathcal{X} \times  \mathcal{Y}$  and then $z\in \mathcal{Z} $ means $z \eqdef (x,y)$, where $x \in \mathcal{X}$,  $y \in \mathcal{Y}$.  When we use $\varphi(z)$, we mean $\varphi(z) = \varphi(x,y)$, and $\varphi(z, \xi) = \varphi(x,y, \xi)$.

For $\mathbf{e} \in \mathcal{RS}^n_2(1)$ (a random vector uniformly distributed on the Euclidean unit sphere) and some constant $\tau$ let $\tilde \varphi(z + \tau \mathbf{e}, \xi) \eqdef \tilde \varphi(x + \tau \mathbf{e}_x, y+ \tau \mathbf{e}_y, \xi)$, where $\mathbf{e}_x$ is the first part of $\mathbf{e}$ size of dimension $n_x \eqdef \text{dim}(x)$, and $\mathbf{e}_y$ is the second part of dimension $n_y \eqdef \text{dim}(y)$. And  $n \eqdef n_x + n_y$. Then define estimation of the gradient through the difference of functions:
\begin{equation}
    \label{eq:PrSt3}
    g(z, \xi, \mathbf{e}) = \frac{n \left(\tilde \varphi(z + \tau \mathbf{e}, \xi) -  \tilde \varphi(z - \tau \mathbf{e}, \xi)\right)}{2\tau}\left(
    \begin{array}{c}
    \mathbf{e}_x\\
    -\mathbf{e}_y \\
    \end{array}
    \right). 
\end{equation}
$g(z, \xi, \mathbf{e})$ is a block vector consisting of two vectors.

Next we define an important object for further theoretical discussion -- a smoothed version of the function $\varphi$ (see \cite{Nesterov}, \cite{Shamir15}).
\begin{definition}
Function $\hat{\varphi}(x, y) = \hat{\varphi}(z)$ defines on set $\mathcal{X}\times \mathcal{Y}$ satisfies:

\begin{equation}
    \label{eq:smoothphi}
    \hat{\varphi}(z) = \mathbb{E}_{\mathbf{e}}\left[\varphi(z+ \tau \mathbf{e}) \right].
\end{equation}
\end{definition}

Note that we introduce a smoothed version of the function only for proof; in the algorithm, we use only the zero-order oracle \eqref{eq:PrSt3}. Now we are ready to present our algorithm:
\begin{algorithm} [H]
	\caption{Zeroth-Order Saddle-Point Algorithm ({\tt zoSPA})}
	\label{alg}
	\begin{algorithmic}
\State 
\noindent {\bf Input:} Iteration limit $N$.
\State Let $z_1 = \argmin\limits_{z \in \mathcal{Z}}d(z)$.
\For {$k=1, 2, \ldots, N$ }
    \State Sample $\mathbf{e}_k$, $\xi_k$ independently.
    \State Initialize $\gamma_k$.
    \State $z_{k+1} =\text{prox}_{z_k}(\gamma_k g(z_k, \xi_k,\mathbf{e}_k))$.
\EndFor
\State 
\noindent {\bf Output:} $\bar z_N$,
\end{algorithmic}
\end{algorithm}
where 
\begin{equation}
    \label{answer}
    \bar z_N = \frac{1}{\Gamma_N} \left(\sum^N_{k = 1} \gamma_k z_k \right), \quad \Gamma_N = \sum^N_{k = 1} \gamma_k .
\end{equation}

In Algorithm \ref{alg}, we use the step $\gamma_k$. In fact, we can take this step as a constant, independent of the iteration number $k$ (see Theorem 1). 



Note that we work only with norms $\| \cdot \|_p$, where $p$ is from 1 to 2 ($q$ is from 2 to $\infty$). In the rest of the paper, including the main theorems, we assume that $p$ is from 1 to 2. 


\begin{lemma}[see Lemma 2 from \cite{Beznosikov}]
\label{beznos}
    For $g(z, \xi,\mathbf{e} )$ defined in \eqref{eq:PrSt3} the following inequalitie holds:
    \begin{eqnarray}
    \label{boundlem0}
        \mathbb{E}\left[ \|g(z, \xi,\mathbf{e} )\|^2_q\right] \leq 2\left(cn M^2 + \frac{n^2\Delta^2}{\tau^2}\right)a^2_q,
    \end{eqnarray}
    where $c$ is some positive constant (independent of $n$) and $a^2_q$ is determined by $\sqrt{\EE[\|e\|_q^4]} \le a^2_q$ and the following statement is true
    \begin{eqnarray}
    a_q^2 = \min\{2q - 1, 32 \log n - 8\} n^{\frac{2}{q} - 1}, \quad \forall n \geq 3.\label{eq:condition_on_u}
\end{eqnarray}
\end{lemma}

\begin{proof}
     Consider the following chain of inequalities, where we use a simple fact \eqref{eq:squared_sum}:
    \begin{eqnarray*}
  \mathbb{E}\left[ \|g(z, \xi,\mathbf{e} )\|^2_q\right]
  &=&\mathbb{E}\left[\left\|\frac{n}{2\tau}\left(\widetilde{\varphi}(z+ \tau \mathbf{e}, \xi) - \widetilde{\varphi}(z - \tau \mathbf{e}, \xi)\right)\mathbf{e}\right\|_q^2\right] \nonumber\\
  &=& \mathbb{E}\left[\left\|\frac{n}{2\tau}\left(\varphi(z + \tau \mathbf{e},\xi) + \delta(z + \tau \mathbf{e}) -  \varphi(z - \tau \mathbf{e},\xi) - \delta(z - \tau \mathbf{e})\right)\mathbf{e}\right\|_q^2\right] \nonumber\\
  &\leq& \frac{n^2}{2\tau^2}\mathbb{E}\left[\left\|\left(\varphi(z + \tau \mathbf{e},\xi) -  \varphi(z - \tau \mathbf{e},\xi)\right)\mathbf{e}\right\|_q^2\right] \nonumber\\
  &&+ \frac{n^2}{2\tau^2}\mathbb{E}\left[\left\|\left(\delta(z + \tau \mathbf{e}) - \delta(z - \tau \mathbf{e})\right)\mathbf{e}\right\|_q^2\right] \nonumber\\
  &\leq& \frac{n^2}{2\tau^2}\mathbb{E}\left[\left(\varphi(z + \tau \mathbf{e},\xi) -  \varphi(z - \tau \mathbf{e},\xi)\right)^2\left\|\mathbf{e}\right\|_q^2\right] \nonumber\\
  &&+ \frac{n^2}{\tau^2}\mathbb{E}\left[\left(\delta^2(z + \tau \mathbf{e}) + \delta^2(z - \tau \mathbf{e})\right)\left\|\mathbf{e}\right\|_q^2\right]
  \end{eqnarray*}
By independence $\xi$,$\mathbf{e}$ and again \eqref{eq:squared_sum} we have
\begin{eqnarray*}
  \mathbb{E}\left[ \|g(z, \xi,\mathbf{e} )\|^2_q\right]
  &\leq& \frac{n^2}{2\tau^2}\mathbb{E}_{\xi}\left[\mathbb{E}_{\mathbf{e}}\left[\left(\varphi(z + \tau \mathbf{e},\xi) - \alpha -  \varphi(z - \tau \mathbf{e},\xi) + \alpha\right)^2\left\|\mathbf{e}\right\|_q^2\right]\right] \nonumber\\
  &&+ \frac{n^2}{\tau^2}\mathbb{E}\left[\left(\delta^2(z + \tau \mathbf{e}) + \delta^2(z - \tau \mathbf{e})\right)\left\|\mathbf{e}\right\|_q^2\right] \nonumber\\
  &\leq& \frac{n^2}{\tau^2}\mathbb{E}_{\xi}\left[\mathbb{E}_{\mathbf{e}}\left[\left(\left(\varphi(z + \tau \mathbf{e},\xi) - \alpha\right)^2 + \left( \varphi(z - \tau \mathbf{e},\xi) - \alpha\right)^2 \right)\left\|\mathbf{e}\right\|_q^2\right]\right] \nonumber\\
  &&+ \frac{n^2}{\tau^2}\mathbb{E}\left[\left(\delta^2(z + \tau \mathbf{e}) + \delta^2(z - \tau \mathbf{e})\right)\left\|\mathbf{e}\right\|_q^2\right]
  \end{eqnarray*}
Taking into account the symmetric distribution of $\mathbf{e}$  and Cauchy--Schwarz inequality:
\begin{eqnarray*}
  \mathbb{E}\left[ \|g(z, \xi,\mathbf{e} )\|^2_q\right]
  &\leq& \frac{2n^2}{\tau^2}\mathbb{E}_{\xi}\left[\mathbb{E}_{\mathbf{e}}\left[\left(\varphi(z + \tau \mathbf{e},\xi) - \alpha\right)^2 \left\|\mathbf{e}\right\|_q^2\right]\right] \nonumber\\
  &&+ \frac{n^2}{\tau^2}\mathbb{E}\left[\left(\delta^2(z + \tau \mathbf{e}) + \delta^2(z - \tau \mathbf{e})\right)\left\|\mathbf{e}\right\|_q^2\right] \nonumber\\
  &\leq& \frac{2n^2}{\tau^2}\mathbb{E}_{\xi}\left[\sqrt{\mathbb{E}_{\mathbf{e}}\left[\left(\varphi(z + \tau \mathbf{e},\xi) - \alpha\right)^4\right]} \sqrt{\mathbb{E}_{\mathbf{e}}\left[\left\|\mathbf{e}\right\|_q^4\right]}\right] \nonumber\\
  &&+ \frac{n^2}{\tau^2}
  \sqrt{\mathbb{E}\left[\left(\delta^2(z + \tau \mathbf{e}) + \delta^2(z - \tau \mathbf{e})\right)^2\right]} \sqrt{\mathbb{E}\left[\left\|\mathbf{e}\right\|_q^4\right]} \nonumber\\
  &\leq& \frac{2n^2 a_q^2}{\tau^2}\mathbb{E}_{\xi}\left[\sqrt{\mathbb{E}_{\mathbf{e}}\left[\left(\varphi(z + \tau \mathbf{e},\xi) - \alpha\right)^4\right]}\right] + \frac{2n^2 a_q^2 \Delta^2}{\tau^2}
  \end{eqnarray*}  
In the last inequality we use \eqref{eq:PrSt2} and \eqref{eq:condition_on_u}. Substituting $\alpha = \mathbb{E}_{\mathbf{e}}\left[\varphi(z + \tau \mathbf{e},\xi)\right]$, applying Lemma \ref{lem:lemma_9_shamir} with the fact that $\varphi(z + \tau \mathbf{e}, \xi)$ is $\tau M(\xi)$-Lipschitz w.r.t. $\mathbf{e}$ in terms of the $\|\cdot\|_2$-norm we get
\begin{eqnarray*}
  \mathbb{E}\left[ \|g(z, \xi,\mathbf{e} )\|^2_q\right]
  &\leq& 2cn a_q^2\cdot \mathbb{E}_{\xi}\left[M^2(\xi)\right] + \frac{2n^2 a_q^2 \Delta^2}{\tau^2} = 2a_q^2 \left(cn\cdot M^2+ \frac{n^2\Delta^2}{\tau^2} \right)
  \end{eqnarray*}
\EndProof
\end{proof}

Note that in the case with $p=2,~ q=2$ we have $a_q = 1$, this follows not from \eqref{eq:condition_on_u}, but from the simplest estimate.  And from \eqref{eq:condition_on_u} we get that  with $p=1,~ q=\infty$ -- $a_q = \mathcal{O}(\nicefrac{\log n}{n})$ (see also Lemma 4 from \cite{Shamir15}).

\begin{lemma}[see Lemma 8 from \cite{Shamir15}]\label{lemma1}
    Let $\mathbf{e}$ be from $\mathcal{RS}^n_2(1)$. Then function $\hat{\varphi}(z, \xi)$ is convex-concave and satisfies :
    \begin{eqnarray}
        \label{lemma1_0}
        \sup_{z \in \mathcal{Z}} |\hat{\varphi}(z) - {\varphi}(z)| \leq \tau M + \Delta. 
    \end{eqnarray}
\end{lemma}
\begin{proof}
    Using definition \eqref{eq:smoothphi} of $\hat \varphi$:
    \begin{eqnarray*}
        \big|\hat{\varphi}(z) - {\varphi}(z)\big|&=& 
        \big|\mathbb{E}_{\mathbf{e}}[\varphi(z + \tau \mathbf{e})] - \varphi(z)\big| \nonumber\\
        &=& \left|\mathbb{E}_{\mathbf{e}}\left[\varphi(z + \tau \mathbf{e}) - \varphi(z) \right] \right|.
    \end{eqnarray*}
    Since $\|\nabla \varphi(z)\|_2 \leq M$, then $\varphi(z)$ is $M$-Lipschitz:
    \begin{eqnarray*}
       \left|\mathbb{E}_{\mathbf{e}}\left[\varphi(z + \tau \mathbf{e}) - \varphi(z) \right]\right| &\leq&  \left|\mathbb{E}_{\mathbf{e}}\left[ M \|\tau \mathbf{e}\|_2 \right] \right| \leq M\tau .
    \end{eqnarray*}
    \EndProof
\end{proof}

\begin{lemma}[see Lemma 10 from \cite{Shamir15} and Lemma 2 from \cite{Beznosikov}]\label{lemmasham}
    It holds that 
    \begin{eqnarray}
    \label{temp985}
    \tilde \nabla \hat{\varphi}(z) &=& \mathbb{E}_{\mathbf{e}}\left[\frac{n\left(\varphi(z + \tau \mathbf{e}) -  \varphi(z - \tau \mathbf{e})\right)}{2\tau}\left(
    \begin{array}{c}
    \mathbf{e}_x\\
    -\mathbf{e}_y \\
    \end{array}\right)\right],\\
    \|\mathbb{E}_{\mathbf{e}}[ g(z, \mathbf{e}) ] - \tilde \nabla \hat{\varphi}(z)\|_q &\leq&  \frac{\Delta n a_q}{\tau}\label{temp986},
    \end{eqnarray}
    where \begin{eqnarray*}
    g(z, \mathbf{e}) 
    &=& \mathbb{E}_{\xi}\left[g(z, \xi, \mathbf{e})\right] \nonumber\\
    &=& \frac{n \left(\tilde \varphi(z + \tau \mathbf{e}) -  \tilde \varphi(z - \tau \mathbf{e})\right)}{2\tau}\left(
    \begin{array}{c}
    \mathbf{e}_x\\
    -\mathbf{e}_y \\
    \end{array}
    \right).
\end{eqnarray*} 
    Hereinafter, by $\tilde \nabla \hat{\varphi}(z)$ we mean a block vector consisting of two vectors $\nabla_x \hat{\varphi}(x,y)$ and $-\nabla_y \hat{\varphi}(x,y)$.
\end{lemma}
\begin{proof}
    The proof of \eqref{temp985} is given in \cite{Shamir15} and follows from the Stokes' theorem.
    Then 
    \begin{eqnarray*}
    \mathbb{E}_{\mathbf{e}}[ g(z, \mathbf{e}) ] - \tilde \nabla \hat{\varphi}(z) &=&  \mathbb{E}_{\mathbf{e}} \left[ \frac{n \left(\delta(z + \tau \mathbf{e}) -  \delta(z - \tau \mathbf{e})\right)}{2\tau}\left(\begin{array}{c}
    \mathbf{e}_x\\
    -\mathbf{e}_y \\
    \end{array}
    \right)\right].
    \end{eqnarray*}
    Using inequalities \eqref{eq:PrSt2} and definition of $a_q$ in Lemma 1 completes the proof.
    \EndProof
\end{proof}

\begin{lemma}[see Lemma 5.3.2 from \cite{nemirovski}]\label{lemma3}
Define $\Delta_k \eqdef g(z_k, \xi_k, \mathbf{e}_k) -  \tilde\nabla \hat{\varphi}(z_k)$.
Let $D(u) \eqdef \sum^N_{k = 1} \gamma_k \langle \Delta_k  ,u - z_{k}\rangle$. Then we have 
\begin{eqnarray}
    \label{19}
    \mathbb{E}\left[\max_{u \in \mathcal{Z}} D(u)\right] &\leq& 
    \Omega^2 + \frac{\Delta \Omega n a_q}{\tau}\sum\limits_{k=1}^N \gamma_k + M^2_{all}\sum\limits_{k=1}^N \gamma_k^2 , 
\end{eqnarray}
where $M^2_{all} \eqdef
2\left(cn M^2 + \frac{n^2\Delta^2}{\tau^2}\right)a^2_q$ is from Lemma 1.
\end{lemma}
\begin{proof} Let define sequence $v$: $v_1 \eqdef z_1$, $v_{k+1} \eqdef \text{prox}_{v_k} (-\rho \gamma_k \Delta_k)$ for some $\rho > 0$:
\begin{eqnarray}
\label{temp177}
    D(u) &=& \sum\limits_{k=1}^N \gamma_k \langle -\Delta_k, z_k - u \rangle \nonumber\\
    &=& \sum\limits_{k=1}^N \gamma_k \langle -\Delta_k, z_k - v_k \rangle + 
    \sum\limits_{k=1}^N \gamma_k \langle -\Delta_k,  v_k - u \rangle . 
\end{eqnarray}
By the definition of $v$ and an optimal condition for the prox-operator, we have for all $u \in \mathcal{Z}$
\begin{eqnarray*}
    \langle -\gamma_k \rho\Delta_k - \nabla d(v_{k})  + \nabla d(v_{k+1}), u - v_{k+1} \rangle \geq 0.
\end{eqnarray*}
Rewriting this inequality, we get
\begin{eqnarray*}
    \langle -\gamma_k \rho\Delta_k, v_k  - u \rangle \leq \langle -\gamma_k \rho\Delta_k, v_k - v_{k+1} \rangle  + \langle \nabla d (v_{k+1}) - \nabla d (v_k), u - v_{k+1} \rangle.
\end{eqnarray*}
Using \eqref{temp111}:
\begin{eqnarray*}
    \langle -\gamma_k \rho\Delta_k, v_k  - u \rangle \leq \langle -\gamma_k \rho\Delta_k, v_k - v_{k+1} \rangle + V_{v_k}(u) -  V_{v_{k+1}}(u) - V_{v_k}(v_{k+1}).
\end{eqnarray*}
Bearing in mind the Bregman divergence property $2V_x(y) \geq \|x-y\|_p^2$:
\begin{eqnarray*}
    \langle -\gamma_k \rho\Delta_k, v_k  - u \rangle \leq \langle -\gamma_k \rho\Delta_k, v_k - v_{k+1} \rangle + V_{v_k}(u) -  V_{v_{k+1}}(u) - \frac{1}{2}\|v_{k+1} - v_k\|_p^2.
\end{eqnarray*}
Using the definition of the conjugate norm:
\begin{eqnarray*}
    \langle -\gamma_k \rho\Delta_k, v_k  - u \rangle &\leq& \|\gamma_k \rho\Delta_k\|_q\cdot\| v_k - v_{k+1} \|_p + V_{v_k}(u) -  V_{v_{k+1}}(u) - \frac{1}{2}\|v_{k+1} - v_k\|_p^2 \nonumber\\ 
    &\leq& \frac{\rho^2 \gamma_k^2}{2}\|\Delta_k\|_q^2 + V_{v_k}(u) -  V_{v_{k+1}}(u).
\end{eqnarray*}
Summing over $k$ from $1$ to $N$:
\begin{eqnarray*}
    \sum\limits_{k=1}^N \gamma_k \rho \langle -\Delta_k,  v_k - u \rangle \leq V_{v_1}(u) -  V_{v_{N+1}}(u) + \frac{\rho^2}{2}\sum\limits_{k=1}^N \gamma_k^2\|\Delta_k\|_q^2.
\end{eqnarray*}
Notice that $V_x(y) \geq 0$ and $V_{v_1}(u) \leq \nicefrac{\Omega^2}{2}$:
\begin{eqnarray}
\label{temp192}
    \sum\limits_{k=1}^N \gamma_k \langle -\Delta_k,  v_k - u \rangle \leq \frac{\Omega^2}{2\rho} + \frac{\rho}{2}\sum\limits_{k=1}^N \gamma_k^2\|\Delta_k\|_q^2.
\end{eqnarray}
Substituting \eqref{temp192} into \eqref{temp177}:
\begin{eqnarray*}
     D(u) &\leq& \sum\limits_{k=1}^N \gamma_k \langle \Delta_k, v_k - z_k  \rangle +
    \frac{\Omega^2}{2\rho} + \frac{\rho}{2}\sum\limits_{k=1}^N \gamma_k^2\|\Delta_k\|_q^2.  
\end{eqnarray*}
The right side is independent of $u$, then
\begin{eqnarray*}
     \max_{u \in \mathcal{Z}} D(u) &\leq& \sum\limits_{k=1}^N \gamma_k \langle \Delta_k, v_k - z_k  \rangle + 
    \frac{\Omega^2}{2\rho} + \frac{\rho}{2}\sum\limits_{k=1}^N \gamma_k^2\|\Delta_k\|_q^2.  
\end{eqnarray*}
Taking the full expectation and 
introducing a new definition $\tilde \Delta_k \eqdef g(z_k, \xi_k, \mathbf{e}_k) -  g(z_k, \mathbf{e}_k)$:
\begin{eqnarray*}
    \mathbb{E}\left[\max_{u \in \mathcal{Z}} D(u)\right] &\leq& \mathbb{E}\left[\sum\limits_{k=1}^N \gamma_k \langle \Delta_k, v_k - z_k \rangle\right] + 
    \frac{\Omega^2}{2\rho} + \frac{\rho}{2}\mathbb{E}\left[\sum\limits_{k=1}^N \gamma_k^2\|\Delta_k\|_q^2\right] \nonumber\\
    &\leq& \mathbb{E}\left[\sum\limits_{k=1}^N \gamma_k \langle \tilde\Delta_k, v_k - z_k \rangle\right] + \mathbb{E}\left[\sum\limits_{k=1}^N \gamma_k \langle \Delta_k - \tilde\Delta_k, v_k - z_k \rangle\right] \nonumber\\ 
       &&+ 
    \frac{\Omega^2}{2\rho} + \frac{\rho}{2}\mathbb{E}\left[\sum\limits_{k=1}^N \gamma_k^2\|\Delta_k\|_q^2\right]
    .  
\end{eqnarray*}
Using the independence of $\mathbf{e}_1, \ldots, \mathbf{e}_N, \xi_1, \ldots, \xi_N$, we have
\begin{eqnarray*}
    \mathbb{E}\left[\max_{u \in \mathcal{Z}} D(u)\right] &\leq& \mathbb{E}\left[\sum\limits_{k=1}^N \gamma_k \mathbb{E}_{\xi_k}\left[\langle \tilde\Delta_k, v_k - z_k \rangle\right]\right] + \mathbb{E}\left[\sum\limits_{k=1}^N \gamma_k \mathbb{E}_{\mathbf{e}_k}\left[\langle \Delta_k - \tilde\Delta_k, v_k - z_k \rangle\right]\right] \nonumber\\ 
       &&+ 
    \frac{\Omega^2}{2\rho} + \frac{\rho}{2}\mathbb{E}\left[\sum\limits_{k=1}^N \gamma_k^2\|\Delta_k\|_q^2\right]
    .  
\end{eqnarray*}
Note that $v_k - z_k$ does not depend on $\mathbf{e}_k$, $\xi_k$ and $\mathbb{E}_{\xi_k} \tilde\Delta_k = 0$. Then 
\begin{eqnarray*}
    \mathbb{E}\left[\max_{u \in \mathcal{Z}} D(u)\right] &\leq& \mathbb{E}\left[\sum\limits_{k=1}^N \gamma_k \langle \mathbb{E}_{\xi_k}\left[\tilde\Delta_k\right], v_k - z_k \rangle\right] + \mathbb{E}\left[\sum\limits_{k=1}^N \gamma_k \langle \mathbb{E}_{\mathbf{e}_k}\left[\Delta_k - \tilde\Delta_k\right], v_k - z_k \rangle\right] \nonumber\\ 
       &&+ 
    \frac{\Omega^2}{2\rho} + \frac{\rho}{2}\mathbb{E}\left[\sum\limits_{k=1}^N \gamma_k^2\|\Delta_k\|_q^2\right]  \nonumber\\ 
    &=& \mathbb{E}\left[\sum\limits_{k=1}^N \gamma_k \langle \mathbb{E}_{\mathbf{e}_k}\left[\Delta_k - \tilde\Delta_k\right], v_k - z_k \rangle\right] + 
    \frac{\Omega^2}{2\rho} + \frac{\rho}{2}\mathbb{E}\left[\sum\limits_{k=1}^N \gamma_k^2\|\Delta_k\|_q^2\right].  
\end{eqnarray*}
By \eqref{temp986} and definition of diameter $\Omega$ we get
\begin{eqnarray*}
    \mathbb{E}\left[\max_{u \in \mathcal{Z}} D(u)\right] &\leq& \frac{\Delta \Omega n a_q}{\tau}\sum\limits_{k=1}^N \gamma_k + \frac{\Omega^2}{2\rho} + \frac{\rho}{2}\sum\limits_{k=1}^N \gamma_k^2\mathbb{E}\left[\|\Delta_k\|_q^2\right].  
\end{eqnarray*}
To prove the lemma, it remains to estimate $\mathbb{E}\left[\|\Delta_k\|_q^2\right]$:
\begin{eqnarray*}
    \mathbb{E}\left[\|\Delta_k\|_q^2\right] &\leq& \mathbb{E}\left[\|g(z_k, \xi_k, \mathbf{e}_k) -  \tilde\nabla\hat{\varphi}(z_k)\|_q^2\right]  \nonumber\\
    &\leq& 2 \mathbb{E}\left[\|g(z_k, \xi_k, \mathbf{e}_k)\|_q^2\right] +  2 \mathbb{E}\left[\|\tilde\nabla\hat{\varphi}(z_k)\|_q^2\right]\nonumber\\
    &\leq& 2 \mathbb{E}\left[\|g(z_k, \xi_k, \mathbf{e}_k)\|_q^2\right] +  2 \mathbb{E}\left[\left\|\frac{n\left(\varphi(z + \tau \mathbf{e}) -  \varphi(z - \tau \mathbf{e})\right)}{2\tau}\mathbf{e}\right\|_q^2\right]
    .  
\end{eqnarray*}
Using Lemma 1,  we have $\mathbb{E}\left[\|\Delta_k\|_q^2\right] \leq 4M^2_{all}$, whence
\begin{eqnarray*}
    \mathbb{E}\left[\max_{u \in \mathcal{Z}} D(u)\right] &\leq& 
    \frac{\Omega^2}{2\rho} + \frac{\Delta \Omega n a_q}{\tau}\sum\limits_{k=1}^N \gamma_k + 2\rho\sum\limits_{k=1}^N \gamma_k^2 M^2_{all}.  
\end{eqnarray*}
Taking $\rho = \nicefrac{1}{2}$ ends the proof.
\EndProof
\end{proof}

We are ready to prove the main theorem.
\begin{theorem} \label{th_main}
Let problem \eqref{problem} with function $\varphi(x,y)$ be solved using Algorithm \ref{alg} with the oracle $g(z_k, \xi_k, \mathbf{e}_k)$ from \eqref{eq:PrSt3}. Assume, that the function $\varphi(x,y)$ and its inexact modification $\widetilde{\varphi}(x,y)$ satisfy the conditions  \eqref{eq:PrSt1}, \eqref{bound}, \eqref{eq:PrSt2}. 
Denote by $N$ the number of iterations. Let step in Algorithm \ref{alg} $\gamma_k = \frac{\Omega}{ M_{all} \sqrt{N}}$.
Then the rate of convergence is given by the following expression
\begin{eqnarray*}
    \mathbb{E}\left[\varepsilon_{sad}(\bar z_N)\right] &\leq&
    \frac{3 M_{all}\Omega}{\sqrt{N}} + \frac{\Delta \Omega n a_q}{\tau}+ 2\tau M,
\end{eqnarray*}
where $\bar z_N$ is defined in \eqref{answer},
$\Omega$ is a diameter of $\mathcal{Z}$, $M^2_{all} = 2\left(cn M^2 +\frac{n^2\Delta^2}{\tau^2}\right)a^2_q$ and 
\begin{equation}
    \label{sad}
    \varepsilon_{sad}(\bar z_N) = \max_{y' \in \mathcal{Y}} \varphi(\bar x_N, y') - \min_{x' \in \mathcal{X}} \varphi(x', \bar y_N), 
\end{equation}
$\bar x_N$, $\bar y_N$ are defined the same way as $\bar z_N$ in \eqref{answer}.
\end{theorem}

\begin{proof} We divided the proof into three steps.

\textbf{Step 1.} Let $g_k \eqdef \gamma_k g(z_k, \xi_k, \mathbf{e}_k)$. By the step of Algorithm \ref{alg}, $z_{k + 1}= \text{prox}_{z_k}(g_k)$. Taking into account \eqref{lemma3_1}, we get that for all $u \in \mathcal{Z}$
\begin{eqnarray*}
    \langle g_k , z_{k+1} - u\rangle = \langle g_k , z_{k+1} - z_{k} + z_{k} - u\rangle \leq V_{z_k}(u) - V_{z_{k + 1}}(u) - V_{z_k}(z_{k + 1}).
\end{eqnarray*}
By simple transformations:
\begin{eqnarray*}
    \langle g_k , z_{k} - u\rangle &\leq& \langle g_k , z_{k} - z_{k+1}\rangle + V_{z_k}(u) - V_{z_{k + 1}}(u) - V_{z_k}(z_{k + 1}) \nonumber\\
    &\leq& \langle g_k , z_{k} - z_{k+1}\rangle + V_{z_k}(u) - V_{z_{k + 1}}(u) - \frac{1}{2}\|z_{k + 1} - z_{k}\|^2_p. 
\end{eqnarray*}
In last inequality we use the property of the Bregman divergence: $V_x(y) \ge \frac{1}{2}\|x-y\|_p^2$.
Using Hölder's inequality and the fact: $ab - \nicefrac{b^2}{2} \leqslant\nicefrac{a^2}{2}$, we have
\begin{eqnarray}
    \label{temp2}
    \langle g_k , z_{k} - u\rangle &\leq&
    \| g_k \|_q\|z_{k} - z_{k+1}\|_p  + V_{z_k}(u) - V_{z_{k + 1}}(u) - \frac{1}{2}\|z_{k + 1} - z_{k}\|^2_p  \nonumber\\
    &\leq&
    V_{z_k}(u) - V_{z_{k + 1}}(u) + \frac{1}{2}\| g_k \|^2_q.
\end{eqnarray}
Summing \eqref{temp2} over all $k$ from 1 to $N$ and by the definitions of $g_k$ and $\Omega$ (diameter of $\mathcal{Z}$):
\begin{equation}
    \label{temp3}
    \sum^N_{k = 1} \gamma_k \langle g(z_k, \xi_k, \mathbf{e}_k) , z_{k} - u\rangle \leq\frac{\Omega^2}{2} + \frac{1}{2}\sum^N_{k = 1} \gamma^2_k \| g(z_k, \xi_k, \mathbf{e}_k) \|^2_q, \quad \forall u \in \mathcal{Z}.
\end{equation}
Substituting the definition of $D(u)$ from\eqref{temp3}, we have for all $u \in \mathcal{Z}$
\begin{eqnarray}
    \label{temp_main}
    \sum^N_{k = 1} \gamma_k \langle \tilde\nabla \hat{\varphi}(z_k) , z_{k} - u\rangle &\leq&\frac{\Omega^2}{2} + \frac{1}{2}\sum^N_{k = 1} \gamma^2_k \| g(z_k, \xi_k, \mathbf{e}_k) \|^2_q + D(u) .
\end{eqnarray}
By $\tilde \nabla \hat{\varphi}(z)$ we mean a block vector consisting of two vectors $\nabla_x \hat{\varphi}(x,y)$ and $-\nabla_y \hat{\varphi}(x,y)$.

\textbf{Step 2.} 
We consider a relationship between functions $\hat{\varphi}(z)$ and $\varphi(z)$. 
Combining \eqref{sad} and \eqref{lemma1_0} we get
\begin{eqnarray*}
    \varepsilon_{sad}(\bar z_N) 
    &\le& \max\limits_{y' \in \mathcal{Y}} \hat{\varphi}(\bar x_N, y') - \min\limits_{x' \in \mathcal{X}} \hat{\varphi}(x', \bar y_N) + 2\tau M. 
\end{eqnarray*}
Then, by the definition of $\bar x_N$ and $\bar y_N$ (see \eqref{sad}), Jensen's inequality and convexity-concavity of $\hat{\varphi}$:
\begin{eqnarray*}
    \varepsilon_{sad}(\bar z_N)
    &\leq& \max\limits_{y' \in \mathcal{Y}} \hat{\varphi}\left(\frac{1}{\Gamma_N} \left(\sum^N_{k = 1} \gamma_k x_k \right), y'\right) - \min\limits_{x' \in \mathcal{X}} \hat{\varphi}\left(x', \frac{1}{\Gamma_N} \left(\sum^N_{k = 1} \gamma_k y_k \right)\right) \nonumber \\
    &&+ 2\tau M
    \nonumber \\
    &\leq& \max\limits_{y' \in \mathcal{Y}} \frac{1}{\Gamma_N} \sum^N_{k = 1} \gamma_k \hat{\varphi}(x_k, y')  - \min\limits_{x' \in \mathcal{X}} \frac{1}{\Gamma_N} \sum^N_{k = 1} \gamma_k\hat{\varphi}(x', y_k) + 2\tau M .
\end{eqnarray*}
Given the fact of linear independence of $x'$ and $y'$:
\begin{eqnarray*}
    \varepsilon_{sad}(\bar z_N) &\leq& \max\limits_{(x', y') \in \mathcal{Z}}\frac{1}{\Gamma_N} \sum^N_{k = 1} \gamma_k \left(\hat{\varphi}(x_k, y')  - \hat{\varphi}(x', y_k) \right) + 2\tau M .
\end{eqnarray*}
Using convexity and concavity of the function $\hat{\varphi}$:
\begin{eqnarray}
\label{temp8}
    \varepsilon_{sad}(\bar z_N) &\leq&  \max\limits_{(x', y') \in \mathcal{Z}}\frac{1}{\Gamma_N} \sum^N_{k = 1} \gamma_k \left(\hat{\varphi}(x_k, y')  - \hat{\varphi}(x', y_k) \right) + 2\tau M  \nonumber \\
    &= & \max\limits_{(x', y') \in \mathcal{Z}} \frac{1}{\Gamma_N} \sum^N_{k = 1} \gamma_k \left(\hat{\varphi}(x_k, y') - \hat{\varphi}(x_k, y_k) + \hat{\varphi}(x_k, y_k) - \hat{\varphi}(x', y_k) \right) \nonumber \\
    &&+ 2\tau M  \nonumber \\
    &\leq& \max\limits_{(x', y') \in \mathcal{Z}} \frac{1}{\Gamma_N} \sum^N_{k = 1} \gamma_k \left(\langle \nabla_y \hat{\varphi} (x_k, y_k), y'-y_k \rangle + \langle \nabla_x \hat{\varphi} (x_k, y_k), x_k-x' \rangle \right) \nonumber \\
    &&+ 2\tau M\nonumber \\
    &\leq& \max\limits_{u \in \mathcal{Z}} \frac{1}{\Gamma_N} \sum^N_{k = 1} \gamma_k \langle \tilde \nabla \hat{\varphi}(z_k), z_k - u\rangle  + 2\tau M.
\end{eqnarray}

\textbf{Step 3.}
Combining expressions \eqref{temp_main} and \eqref{temp8}, we get
\begin{eqnarray*}
    \varepsilon_{sad}(\bar z_N)
    &\leq& \frac{\Omega^2}{2\Gamma_N} + \frac{1}{2\Gamma_N}\sum^N_{k = 1} \gamma^2_k \mathbb{E}\left[\| g(z_k, \xi_k, \mathbf{e}_k) \|^2_q\right] + \frac{1}{\Gamma_N}\max_{u \in \mathcal{Z}} D(u) \nonumber \\ 
    && + 2\tau M.
\end{eqnarray*}
Taking full mathematical expectation and using \eqref{19} and \eqref{boundlem0} , we have
\begin{eqnarray*}
    \mathbb{E}\left[\varepsilon_{sad}(\bar z_N)\right] &\leq&
    \frac{3\Omega^2}{2 \Gamma_N} + \frac{3M^2_{all}}{2\Gamma_N}\sum^N_{k = 1} \gamma^2_k  + \frac{\Delta \Omega n a_q}{\tau} + 2\tau M .
\end{eqnarray*} 
Substituting values of $\gamma_k = \frac{\Omega}{M_{all} \sqrt{N}}$ ends the proof.
\EndProof
\end{proof}

Next we analyze the results.
\begin{corollary} Under the assumptions of the Theorem 1 let $\varepsilon$ be accuracy of the solution of the problem \eqref{problem} obtained using Algorithm \ref{alg}. Assume that
\begin{eqnarray}
    \label{temp1209}
    \tau = \Theta \left( \frac{\varepsilon}{M}\right),\quad \Delta = \mathcal{O} \left(\frac{\varepsilon^2 }{M \Omega n a_q}\right),
\end{eqnarray}
then the number of iterations to find $\varepsilon$-solution
\begin{eqnarray*}
    N = \mathcal{O} \left( \frac{\Omega^2 M^2 n^{2/q}}{\varepsilon^2}C^2(n,q)\right),
\end{eqnarray*}
where $C(n,q) \eqdef \min\{2q - 1, 32 \log n - 8\}$.
\end{corollary}
\begin{proof}
    From the conditions of the corollary it follows:
    \begin{eqnarray*}
    M^2_{all} = \Theta(n M^2a_q^2) = \Theta(n^{2/q} M^2 C(n,q)).
\end{eqnarray*}
This immediately implies the assertion of the corollary.
\EndProof
\end{proof}

Consider separately cases with $p = 1$ and $p = 2$.
\begin{table}[H]
    \centering
    \begin{tabular}{ |c | c | c |  }
    \hline
    $p$, ($1\leqslant p \leqslant 2$) & $q$,  ($2\leqslant q \leqslant \infty$) &   $N$, Number of iterations\\ \hline
    $p = 2$& $q = 2$ & $\mathcal{O}\left(\frac{\Omega^2 M^2}{\varepsilon^2}n\right)$\\\hline
    $p = 1$& $q = \infty$  & $\mathcal{O}\left(\frac{\Omega^2 M^2}{\varepsilon^2}\log^2(n)\right)$\\ \hline
\end{tabular}
    \caption{Summary of  convergence estimation for non-smooth case: $p = 2$ and $p = 1$.}
    \label{summary_estim}
\end{table}

Note that in the case with $p = 2$, we have that the number of iterations increases $n$ times compared with \cite{nemirovski}, and in the case with $p=1$ -- just $\log^2 n$ times.

\subsection{Admissible Set Analysis}\label{ASA}

As stated above, in works (see \cite{Shamir15}, \cite{duchi2013optimal}), where zeroth-order approximation (\eqref{eq:PrSt3}) is used instead of the "honest"{} gradient, it is important that the function is specified not only on an admissible set, but in a certain neighborhood of it. This is due to the fact that for any point $x$ belonging to the set, the point $x + \tau \mathbf{e}$ can be outside it.

But in some cases we cannot make such an assumption. The function and values of $x$ can have a real physical interpretation. For example, in the case of a probabilistic simplex, the values of $x$ are the distribution of resources or actions. The sum of the probabilities cannot be negative or greater than $1$. Moreover, due to implementation or other reasons, we can deal with an oracle that is clearly defined on an admissible set and nowhere else.

In this part of the paper, we outline an approach how to solve the problem raised above and how the quality of the solution changes from this.

Our approach can be briefly described as follows:
\begin{itemize}
    \item  Compress our original set $X$ by $(1-\alpha)$ times and consider a "reduced"{} version $X^{\alpha}$. Note that the parameter $\alpha$ should not be too small, otherwise the parameter $\tau$ must be taken very small. But it’s also impossible to take large $\alpha$, because we compress our set too much and can get a solution far from optimal. This means that the accuracy of the solution $\varepsilon$ bounds $\alpha$: $\alpha \leq h(\varepsilon)$, in turn, $\alpha$ bounds $\tau$: $\tau \leq g(\alpha)$.
    \item Generate a random direction $\mathbf{e}$ so that for any $x \in X^{\alpha}$ follows $x + \tau \mathbf{e} \in X$. 
    \item Solve the problem on "reduced"{} set with $\nicefrac{\varepsilon}{2}$-accuracy. The $\alpha$ parameter must be selected so that we find $\varepsilon$-solution of the original problem.
\end{itemize}

In practice, this can be implemented as follows: 1) do as described in the previous paragraph, or 2) work on the original set $X$, but if $x_k + \tau \mathbf{e}$ is outside $X$, then project $x_k$ onto the set $X^\alpha$. We provide a theoretical analysis only for the method that always works on $X^\alpha$.

Next, we analyze cases of different sets. General analysis scheme:
\begin{itemize}
    \item Present a way to "reduce"{} the original set.
    \item Suggest a random direction $\mathbf{e}$ generation strategy.
    \item Estimate the minimum distance between $X^\alpha$ and $X$ in $\ell_2$-norm. This is the border of $\tau$, since $\|\mathbf{e}\|_2$.
    \item Evaluate the $\alpha$ parameter so that the $\nicefrac{\varepsilon}{2}$-solution of the "reduced"{} problem does not differ by more than $\nicefrac{\varepsilon}{2}$ from the $\varepsilon$-solution of the original problem.
\end{itemize}

The first case of set is a \textbf{probability simplex}:

\begin{eqnarray*}
\triangle_n = \left\{\sum\limits_{i=1}^n x_i = 1, \quad  x_i \geq 0, \quad i \in 1\ldots n\right\}.
\end{eqnarray*}
Consider the hyperplane
\begin{eqnarray*}
\mathcal{H} = \left\{\sum\limits_{i=1}^n x_i = 1\right\},
\end{eqnarray*}
in which the simplex lies. Note that if we take the directions $\mathbf{e}$ that lies in $\mathcal{H}$, then for any $x$ lying on this hyperplane, $x + \tau \mathbf{e}$ will also lie on it. Therefore, we generate the direction $\mathbf{e}$ randomly on the hyperplane. Note that $\mathcal{H}$ is a subspace of $\mathbb{R}^n$ with size $\text{dim}\mathcal{H} = n-1$. One can check that the set of vectors from $\mathbf{R}^n$
\begin{eqnarray*}
\mathbf{v} = \left(
\begin{array}{l}
    \mathbf{v}_1 = \nicefrac{1}{\sqrt{2}}(1,-1,0,0,\ldots0),\\
    \mathbf{v}_2 = \nicefrac{1}{\sqrt{6}}(1,1,-2,0,\ldots0),\\
    \mathbf{v}_3 = \nicefrac{1}{\sqrt{12}}(1,1,1,-3,\ldots0),\\
    \ldots \\
    \mathbf{v}_k = \nicefrac{1}{\sqrt{k + k^2}} (1,\ldots1,-k,\ldots,0),\\
    \ldots\\
    \mathbf{v}_{n-1} = \nicefrac{1}{\sqrt{n-1 + (n-1)^2}} (1,\ldots,1,-n+1)
\end{array}
\right),
\end{eqnarray*}
is an orthonormal basis of $\mathcal{H}$. Then generating the vectors $ \mathbf{\tilde e}$ uniformly on the euclidean sphere $\mathcal{RS}^{n-1}_2(1)$ and computing $ \mathbf{e}$ by the following formula:
\begin{eqnarray}
\label{e}
\mathbf{e} =  \mathbf{\tilde e}_1 \mathbf{v}_1 + \mathbf{\tilde e}_2 \mathbf{v}_2 + \ldots + \mathbf{\tilde e}_k \mathbf{v}_k + \ldots  \mathbf{\tilde e}_{n-1} \mathbf{v}_{n-1},
\end{eqnarray}
we have what is required. With such a vector $\mathbf{e}$, we always remain on the hyperplane, but we can go beyond the simplex. This happens if and only if for some $i$, $x_i + \tau \mathbf{e}_i <0$. To avoid this, we consider a "reduced"{} simplex for some positive constant $\alpha$:
\begin{eqnarray*}
\triangle^\alpha_n = \left\{\sum\limits_{i=1}^n x_i = 1, \quad  x_i \geq \alpha, \quad i \in 1\ldots n\right\}.
\end{eqnarray*}
One can see that for any $x \in \triangle^\alpha_n$, for any $\mathbf{e}$ from \eqref{e} and $\tau < \alpha$ follows that $x+\tau \mathbf{e} \in \triangle_n$, because $|\mathbf{e}_i| \leq 1$ and then $x_i + \tau \mathbf{e}_i \geq \alpha - \tau \geq 0$.

The last question to be discussed is the accuracy of the solution that we obtain on a "reduced"{} set. Consider the following lemma (this lemma does not apply to the problem \eqref{problem},for it we prove later):

\begin{lemma}\label{lemm_clos}
    Suppose the function $f(x)$ is $M$-Lipschitz w.r.t. norm $\|\cdot\|_2$. Consider the problem of minimizing $f(x)$ not on original set $X$, but on the "reduced"{} set $X_\alpha$. Let we find $x_k$ solution with $\nicefrac{\varepsilon}{2}$-accuracy on $f(x)$. Then we found $\left(\nicefrac{\varepsilon}{2} + r M\right)$-solution of original problem, where
    \begin{eqnarray*}
    r = \max_{x \in X} \left\|x - \argmin_{\hat x \in X^{\alpha}} \|x-\hat x\|_2\right\|_2.
    \end{eqnarray*}
\end{lemma}

\begin{proof} Let $x^*$ is a solution of the original problem and $\hat x$ is a point of $X_\alpha$, closest to $x^*$. 
    We use the fact that $f (\tilde x) \leq f (\hat x)$ and $M$-Lipschitz continuity of $f$:
    \begin{eqnarray*}
    f(x_k) - f(x^*) &=& f(x_k) - f(\hat x) + f(\hat x) - f(x^*) \nonumber\\ 
    &\leq& f(x_k) - f(\tilde x) + f(\hat x) - f(x^*) \leq \frac{\varepsilon}{2} + M \|\hat x - x^* \|_2 \nonumber\\
    &\leq&  \frac{\varepsilon}{2} + M r.
    \end{eqnarray*}
\EndProof
\end{proof}
It is not necessary to search for the closest point to each $x$ and find $r$. It’s enough to find one that is "pretty"{} close and find some upper bound of $r$. Then it remains to find a rule, which each point $x$ from $X$ associated with some point $\hat x$ from $X_\alpha$ and estimate the maximum distance $\max_X \|\hat x - x\|_2 $.
For any simplex point, consider the following rule: 
\begin{equation*}
    \hat x_i = \frac{(x_i + 2\alpha)}{(1 + 2\alpha n)}, \qquad i = 1, \ldots n.
\end{equation*}
One can easy to see, that for $\alpha \leq \nicefrac{1}{2n}$:
\begin{equation*}
    \sum\limits_{i=1}^n\hat x_i = 1, \qquad \hat x_i \leq \alpha, \qquad i = 1, \ldots n.
\end{equation*}
It means that $\hat x \in X_\alpha$. The distance $\|\hat x - x\|_2 $:
\begin{equation*}
    \|\hat x - x\|_2 = \sqrt{\sum\limits_{i=1}^n (\hat x_i - x_i)^2} = \frac{2 \alpha n}{1 + 2\alpha n} \sqrt{\sum\limits_{i=1}^n \left(\frac{1}{n} - x_i\right)^2}.
\end{equation*}
$\sqrt{\sum\limits_{i=1}^n \left(\frac{1}{n} - x_i\right)^2}$  is a distance to the center of the simplex. It can be bounded by the radius of the circumscribed sphere $R = \sqrt{\frac{n-1}{n}} \leq 1$. Then 
\begin{equation}
    \label{temp109}
    \|\hat x - x\|_2 \leq \frac{2 \alpha n}{1 + 2\alpha n} \leq 2 \alpha n.
\end{equation}
\eqref{temp109} together with Lemma \ref{lemm_clos} gives that $f(x_k) - f(x^*) \leq \frac{\varepsilon}{2} + 2\alpha nM$. Then by taking $\alpha = \nicefrac{\varepsilon}{4nM}$ (or less), we find $\varepsilon$-solution of the original problem. And it takes $\tau \leq \alpha = \nicefrac{\varepsilon}{4nM}$.

The second case is a \textbf{cube}:

\begin{eqnarray*}
\mathcal{C}_n = \left\{l_i\leq x_i \leq u_i, \quad i \in 1\ldots n\right\}.
\end{eqnarray*}
We propose to consider a "reduced"{} set of the following form:
\begin{eqnarray*}
\mathcal{C}^{\alpha}_n = \left\{l_i + \alpha\leq y_i \leq u_i - \alpha, \quad i \in 1\ldots n\right\}.
\end{eqnarray*}
One can note that for all $i$ the minimum of the expression  $y_i + \tau\mathbf{e}_i$ is equal to $l_i + \alpha - \tau$ (maximum -- $u_i - \alpha + \tau$), because $-1 \leq \mathbf{e}_i \leq 1$. Therefore, it is necessary that $l_i + \alpha - \tau \geq l_i$ and $u_i - \alpha  +\tau \leq u_i$. It means that for $\alpha \geq \tau$ and any $\mathbf{e} \in \mathcal{RS}^n_2(1)$, for the vector $y + \tau\mathbf{e}$ the following expression is valid:
\begin{eqnarray*}
l_i \leq y_i + \tau \mathbf{e}_i \leq u_i, \quad i \in 1\ldots n.
\end{eqnarray*}

Then let find $r$ in Lemma \ref{lemm_clos} for cube. Let for any $x \in \mathcal{C}_n$ define $\hat x$ in the following way:
\begin{equation*}
    \hat x_i = 
    \begin{cases}
    l_i + \alpha, \quad x_i < l_i + \alpha,\\
    x_i, \quad l_i + \alpha \leq  x_i \leq u_i - \alpha, \\
    u_i - \alpha, \quad x_i \geq u_i - \alpha,
    \end{cases} \qquad i = 1, \ldots n.
\end{equation*}
One can see that $\hat x_i \in \mathcal{C}^{\alpha}_n$ and
\begin{equation*}
    \|\hat x - x\|_2 = \sqrt{\sum\limits_{i=1}^n (\hat x_i - x_i)^2} \leq \sqrt{\sum\limits_{i=1}^n \alpha^2} = \alpha \sqrt{n}.
\end{equation*}
By Lemma \ref{lemm_clos} we have that $f(x_k) - f(x^*) \leq \frac{\varepsilon}{2} + \alpha \sqrt{n}M$. Then by taking $\alpha = \nicefrac{\varepsilon}{2\sqrt{n}M}$ (or less), we find $\varepsilon$-solution of the original problem. And it takes $\tau \leq \alpha = \nicefrac{\varepsilon}{2\sqrt{n}M}$.

The third case is a \textbf{ball in $p$-norm} for $p \in [1;2]$:
\begin{eqnarray*}
\mathcal{B}^n_p(a, R) = \left\{\|x - a\|_p \leq R\right\},
\end{eqnarray*}
where $a$ is a center of ball, $R$ -- its radii. We propose reducing a ball and solving the problem on the "reduced"{} ball $\mathcal{B}^n_p(a, R(1 - \alpha))$. We need the following lemma:

\begin{lemma}
    Consider two concentric spheres in $p$ norm, where $p \in [1;2]$, $\alpha \in (0;1)$:
    \begin{eqnarray*}
    \mathcal{S}^n_p(a, R) = \left\{\|x - a\|_p = R\right\}, \quad \mathcal{S}^n_p(a, R(1-\alpha)) = \left\{\|y - a\|_p = R(1 - \alpha)\right\}.
    \end{eqnarray*}
    Then the minimum distance between these spheres in the second norm
    \begin{eqnarray*}
    m = \frac{\alpha R}{n^{\nicefrac{1}{p} -\nicefrac{1}{2}}}.
    \end{eqnarray*}
\end{lemma}
\begin{proof} Without loss of generality, one can transfer the center of the spheres to zero, and also, by virtue of symmetry, consider only parts of the spheres, where all components are positive. Then rewriting the problem of finding the minimum distance we get
    \begin{eqnarray}
    \label{temp344}
    \min_{x,y \in \mathbb{R}^n_+} & &\|x-y\|_2 \\
    \text{s.t.} & &x \in \mathcal{S}^n_p(R, 0),  \nonumber\\
    &&y \in \mathcal{S}^n_p(R(1 -\alpha), 0). \nonumber
    \end{eqnarray}
Lagrange function of \eqref{temp344}:
\begin{eqnarray*}
    \label{temp345}
    L = \sum\limits_{i=1}^n (x_i - y_i)^2 + \lambda_1 \left(\sum\limits_{i=1}^n x_i^p - R^p\right) + \lambda_2 \left(\sum\limits_{i=1}^n y_i^p - (1 - \alpha)^pR^p\right).
\end{eqnarray*}
Note that we don not add restrictions for $x_i \geq 0$ and $y_i \geq 0$ into the Lagrange function.

Taking derivatives with respect to $x_i$ and $y_i$ and using the necessary conditions for the extremum point:
\begin{eqnarray}
    \label{temp346}
    \begin{cases}
   L_{x_i} = 2 (x_i - y_i) + \lambda_1 p x^{p-1}_i = 0,\\
   L_{y_i} = 2 (y_i - x_i) + \lambda_2 p y^{p-1}_i = 0,\\
   \sum\limits_{i=1}^n x_i^p - R^p = 0 ,\\
   \sum\limits_{i=1}^n y_i^p - (1 - \alpha)^p R^p = 0.
 \end{cases}
\end{eqnarray}
One can note that $x_i > y_i$, hence $\lambda_1 < 0$, $\lambda_2 > 0$.
From first two equations of \eqref{temp346}:
\begin{eqnarray}
\label{temp347}
   -\lambda_1 p x^{p-1}_i &=& \lambda_2 p y^{p-1}_i \\
   (-\lambda_1)^{\nicefrac{p}{p-1}}  x^{p}_i&=& \lambda_2^{\nicefrac{p}{p-1}}  y^{p}_i \nonumber\\
   (-\lambda_1)^{\nicefrac{p}{p-1}} \sum\limits_{i=1}^n x^{p}_i &=& \lambda_2^{\nicefrac{p}{p-1}} \sum\limits_{i=1}^n  y^{p}_i \nonumber\\
   (-\lambda_1)^{\nicefrac{p}{p-1}} R^p &=& \lambda_2^{\nicefrac{p}{p-1}} (1 - \alpha)^p R^p \nonumber\\
   -\lambda_1 &=& \lambda_2 (1 - \alpha)^{p-1}.
   \label{temp348}
\end{eqnarray}
Combining \eqref{temp347} and \eqref{temp348}, we have
\begin{eqnarray*}
(1 - \alpha)x_i = y_i.
\end{eqnarray*}
Substituting $y_i$ from the first equation of \eqref{temp346} into the second equation of \eqref{temp346} and using \eqref{temp347},  we get
\begin{eqnarray*}
-\lambda_1 x_i^{p-1} &=& \lambda_2  \left(x_i + \frac{\lambda_1 p}{2}x_i^{p-1}\right)^{p-1} \nonumber\\
(1 - \alpha)^{p-1} x_i^{p-1} &=& \left(x_i + \frac{\lambda_1 p}{2}x_i^{p-1}\right)^{p-1}\nonumber\\
(1 - \alpha) x_i &=& x_i + \frac{\lambda_1 p}{2}x_i^{p-1} \nonumber\\
 - \alpha &=& \frac{\lambda_1 p}{2}x_i^{p-2} \nonumber\\
  \left(\frac{-2\alpha}{\lambda_1 p}\right)^{\nicefrac{p}{p-2}} &=& x_i^{p} \nonumber\\
 n \left(\frac{-2\alpha}{\lambda_1 p}\right)^{\nicefrac{p}{p-2}} &=& R^{p} \nonumber\\
 \lambda_1 &=& \frac{-2 \alpha n^{\nicefrac{p-2}{p}}}{p R^{p-2}}.
\end{eqnarray*}
Then by \eqref{temp348}:
\begin{eqnarray*}
\label{temp351}
\lambda_2 = \frac{2 \alpha n^{\nicefrac{p-2}{p}}}{(1 - \alpha)^{p-1} p R^{p-2}}.
\end{eqnarray*}
Substituting $\lambda_1$, $x_i - y_i = \alpha x_i$ into the first equation of \eqref{temp346}:
\begin{eqnarray*}
\label{temp352}
\alpha x_i = \frac{2 \alpha n^{\nicefrac{p-2}{p}} p}{2p R^{p-2}} x_i^{p-1}.
\end{eqnarray*}
Whence it follows that
\begin{eqnarray*}
\label{temp353}
x_i = \frac{R}{n^{\nicefrac{1}{p}}}, \quad y_i = \frac{(1 - \alpha)R}{n^{\nicefrac{1}{p}}}.
\end{eqnarray*}
The values we found are non-negative. Then it is very easy to get the value of $m$:
\begin{eqnarray*}
\label{temp354}
m = \sqrt{\sum\limits_{i=1}^n (x_i - y_i)^2} = \frac{\alpha R}{n^{\nicefrac{1}{p} - \nicefrac{1}{2}}}.
\end{eqnarray*}
It remains only to verify that the found value is a minimum. Taking into account that $x_i = x_j$, $y_i = y_j$ and $y_i = (1 - \alpha) x_i$, one can write $dx_i = dx_j$, $y_i = y_j$, $dy_i = (1 - \alpha) dx_i$ and find $d^2 L$:
\begin{eqnarray*}
\label{temp355}
d^2 L &=& \sum\limits_{i=1}^n L_{x_i x_i}(dx_i)^2 + 2\sum\limits_{i=1}^n L_{x_i y_i}dx_i dy_i  + \sum\limits_{i=1}^n L_{y_i y_i}(dy_i)^2 \\
&=& n L_{x_1 x_1}(dx_1)^2 + 2n L_{x_1 y_1}dx_i dy_i + n L_{y_1 y_1}(dy_1)^2\\
&=& n (L_{x_1 x_1} + 2(1 - \alpha) L_{x_1 y_1} + (1 - \alpha)^2 L_{y_1 y_1})  (dx_1)^2\\
&=& n \left(2 - 2\alpha(p-1) - 4 (1 - \alpha) + (1 - \alpha) (2 + 2\alpha(p-1))\right) (dx_1)^2\\
&=& 2n\alpha(1 - \alpha (p-1))(dx_1)^2.
\end{eqnarray*}
For $\alpha \in (0;1)$ and $p \in [1;2]$ we have $\alpha (p-1) \leq 0$, hence $d^2L \geq 0$. It means that we find a minimum of the distance.
\EndProof
\end{proof}
Using the lemma, one can see that for any $x \in \mathcal{B}^\alpha_n(a,R(1 - \alpha))$, $\tau \leq \nicefrac{\alpha R}{n^{\nicefrac{1}{p} - \nicefrac{1}{2}}}$ and for any $\mathbf{e} \in \mathcal{RS}^n_2(1)$, $x + \tau \mathbf{e} \in \mathcal{B}_n(a, R)$.

Then let find $r$ in Lemma \ref{lemm_clos} for ball. Let for any $x$ define $\hat x$ in the following way:
\begin{equation*}
    \hat x_i = a + (1 - \alpha)(x_i - a), \qquad i = 1, \ldots n.
\end{equation*}
One can see that $\hat x_i$ is in the "reduced"{} ball and
\begin{equation*}
    \|\hat x - x\|_2 = \sqrt{\sum\limits_{i=1}^n (\hat x_i - x_i)^2} = \sqrt{\sum\limits_{i=1}^n (\alpha (x_i - a))^2} = \alpha \sqrt{\sum\limits_{i=1}^n  (x_i-a)^2} \leq \alpha \sum\limits_{i=1}^n  |x_i-a|.
\end{equation*}
By Holder inequality:
\begin{equation*}
    \|\hat x - x\|_2 \leq \alpha \sum\limits_{i=1}^n  |x_i - a| \leq \alpha n^{\frac{1}{q}} \left(\sum\limits_{i=1}^n  |x_i - a|^p\right)^{\frac{1}{p}} = \alpha n^{\frac{1}{q}}R.
\end{equation*}

By Lemma \ref{lemm_clos} we have that $f(x_k) - f(x^*) \leq \frac{\varepsilon}{2} + \alpha n^{1/q}RM$. Then by taking $\alpha = \nicefrac{\varepsilon}{2n^{1/q}RM}$ (or less), we find $\varepsilon$-solution of the original problem. And it takes $\tau \leq \nicefrac{\alpha R}{n^{\nicefrac{1}{p} - \nicefrac{1}{2}}} = \nicefrac{\varepsilon}{2M \sqrt{n}}$.

The fourth case is a \textbf{product of sets $\mathcal{Z} = \mathcal{X} \times \mathcal{Y}$}. We define the "reduced"{} set $Z^{\alpha}$ as 
\begin{eqnarray*}
Z^\alpha = X^\alpha \times Y^\alpha,
\end{eqnarray*}
We need to find how the parameter $\alpha$ and $\tau$ depend on the parameters $\alpha_x$, $\tau_x$ and $\alpha_y$, $\tau_y$ for the corresponding sets $X$ and $Y$, i.e. we have bounds: $\alpha_x \leq h_x(\varepsilon)$, $\alpha_y \leq h_y(\varepsilon)$ and $\tau_x \leq g_x(\alpha_x) \leq g_x(h_x(\varepsilon))$, $\tau_y \leq g_y(\alpha_y) \leq g_y(h_y(\varepsilon))$. Obviously, the functions $g$, $h$ are monotonically increasing for positive arguments. This follows from the physical meaning of $\tau$ and $\alpha$.

Further we are ready to present an analogue of Lemma \ref{lemm_clos}, only for the saddle-point problem.

\begin{lemma}\label{lemm_clos1}
    Suppose the function $\varphi(x,y)$ inthe  saddle-point problem is $M$-Lipschitz. Let we find $(\tilde x, \tilde y)$ solution on $X^\alpha$ and $Y^\alpha$ with $\nicefrac{\varepsilon}{2}$-accuracy. Then we found $\left(\nicefrac{\varepsilon}{2} + (r_x + r_y) M\right)$-solution of the original problem, where $r_x$ and $r_y$ we define in the following way:
    \begin{eqnarray*}
    r_x = \max_{x \in X} \left\|x - \argmin_{\hat x \in X^{\alpha}} \|x-\hat x\|_2\right\|_2, \nonumber\\
    r_y = \max_{y \in Y} \left\|y - \argmin_{\hat y \in Y^{\alpha}} \|y-\hat y\|_2\right\|_2.
    \end{eqnarray*}
\end{lemma}
\begin{proof}
    Let $x^* = \argmin_{x \in X} \varphi(x,y_k)$, $y^* = \argmax_{y \in Y} \varphi(x_k,y)$, $\hat x \in X^{\alpha}$ is the closest point to $x^*$ and $\hat y \in Y^{\alpha}$ -- to $y^*$. We use $M$-Lipschitz continuity of $\varphi$:
    \begin{eqnarray*}
    \max_{y \in Y} \varphi(x_k,y) - \min_{x \in X} \varphi(x,y_k) &=& \varphi(x_k,y^*) - \varphi(x_k,\hat y) + \varphi(x_k,\hat y) \nonumber\\
    &&- \varphi(x^*,y_k) + \varphi(\hat x,y_k) - \varphi(\hat x,y_k) \nonumber\\
    &=& \varphi(x_k,\hat y) - \varphi(\hat x,y_k)\nonumber\\
    &&+ \varphi(x_k,y^*) - \varphi(x_k,\hat y)\nonumber\\
    &&+ \varphi(\hat x,y_k) -\varphi(x^*,y_k)\nonumber\\
    &\leq& \max_{y \in Y_\alpha}\varphi(x_k, y) - \min_{x \in X_\alpha}\varphi( x,y_k)\nonumber\\
    &&+ r_x M + r_y M \nonumber\\
    &\leq& \frac{\varepsilon}{2} + (r_x + r_y)M.
    \end{eqnarray*}
\EndProof
\end{proof}
In the previous cases we found the upper bound $\alpha_x \leq h_x(\varepsilon)$ from the condition that $r_xM \leq \nicefrac{\varepsilon}{2}$. Now let's take $\tilde \alpha_x$ and $\tilde \alpha_y$ so that $r_xM \leq \nicefrac{\varepsilon}{4}$ and $r_y M \leq \nicefrac{\varepsilon}{4}$. For this we need $\tilde \alpha_x \leq h_x(\varepsilon/2)$, $\tilde \alpha_y \leq h_y(\varepsilon/2)$. It means that if we take $\alpha = \min(\tilde \alpha_x, \tilde \alpha_y)$, then $(r_x + r_y)M \leq \nicefrac{\varepsilon}{2}$ for such $\alpha$. For a simplex, a cube and a ball the function $h$ is linear, therefore the formula turns into a simpler expression: $\alpha = \nicefrac{\min(\alpha_x, \alpha_y)}{2}$.

For the new parameter $\alpha = \min(\tilde \alpha_x, \tilde \alpha_y)$, we find $\tilde \tau_x = g_x(\alpha) = g_x(\min(\tilde \alpha_x, \tilde \alpha_y))$ and $\tilde \tau_y = g_y(\alpha) = g_y(\min(\tilde \alpha_x, \tilde \alpha_y))$. Then for any $x \in X^{\alpha}$, $\mathbf{e}_x \in \mathcal{RS}^{\text{dim}X}_2(1)$, $y \in Y^{\alpha}$, $\mathbf{e}_y \in \mathcal{RS}^{\text{dim}Y}_2(1)$, $x+\tilde \tau_x \mathbf{e}_x \in X$ and $y+\tilde \tau_y \mathbf{e}_y \in Y$. Hence, it is easy to see that for $\tau = \min(\tilde \tau_x, \tilde \tau_y)$ and the vector $\mathbf{\tilde e}_x$ of the first $\text{dim}X$ components of $\mathbf{e} \in \mathcal{RS}^{\text{dim}X + \text{dim}Y}_2(1)$ and for the vector $\mathbf{\tilde e}_y$ of the remaining $\text{dim}Y$ components, for any $x \in X^{\alpha}$, $y \in Y^{\alpha}$ it is true that $x + \tau \mathbf{\tilde e}_x \in X$ and $y + \tau \mathbf{\tilde e}_y \in Y$. We get $\tau = \min(\tilde \tau_x, \tilde \tau_y)$. In the previous cases that we analyzed (simplex, cube and ball), the function $g$ and $h$ are linear therefore the formula turns into a simpler expression: $\tau = \min(\alpha_x, \alpha_y) \cdot \min(\nicefrac{\tau_x}{\alpha_x}, \nicefrac{\tau_y}{\alpha_y})/2$.

Summarize the results of this part of the paper in Table \ref{summary_contr}.
\begin{table}[h!]
    \centering
    \begin{tabular}{ |c | c | c | c | c | }
    \hline
    Set &  $\alpha$ of "reduced"{} set & Bound of $\tau$ & $\mathbf{e}$
    \\ \hline
    \specialcell{probability\\ simplex} & $\frac{\varepsilon}{4nM}$ & $\frac{\varepsilon}{4nM}$ & see \eqref{e}
    \\ \hline
    \specialcell{cube} & $\frac{\varepsilon}{2\sqrt{n}M}$ & $\frac{\varepsilon}{2\sqrt{n}M}$ & $\mathcal{RS}^n_2(1)$
    \\ \hline
    \specialcell{ball in \\ $p$-norm} & $\frac{\varepsilon}{2n^{1/q}RM}$ & $\frac{\varepsilon}{2\sqrt{n}M}$ & $\mathcal{RS}^n_2(1)$
    \\ \hline
    \specialcell{$X^{\alpha} \times Y^{\alpha}$} & $\frac{\min(\alpha_x, \alpha_y)}{2}$ & $\frac{\min(\alpha_x, \alpha_y) \cdot \min(\nicefrac{\tau_x}{\alpha_x}, \nicefrac{\tau_y}{\alpha_y})}{2}$ & $\mathcal{RS}^n_2(1)$
    \\ \hline
\end{tabular}
    \caption{Summary of the part \ref{ASA}}
    \label{summary_contr}
\end{table}

One can note that in \eqref{temp1209} $\tau$ is independent of $n$. According to Table \ref{summary_contr}, we need to take into account the dependence on $n$. In Table \ref{summary_contr1}, we present the constraints on $\tau$ and $\Delta$ so that Corollary 1 remains satisfied. We consider three cases when both sets $X$ and $Y$ are simplexes, cubes and balls with the same dimension $n/2$. 

The second column of Table 3 means whether the functions are defined not only on the set itself, but also in some neighbourhood of it.

\begin{table}[h!]
    \centering
    \begin{tabular}{ |c | c | c | c | }
    \hline
    Set &  Neigh-d? & $\tau$ & $\Delta$ 
    \\ \hline
    \multirow{2}{*}{\specialcell{probability\\ simplex}} &  
    \cmark & $\Theta \left( \frac{\varepsilon}{M}\right)$ & $\mathcal{O} \left(\frac{\varepsilon^2 }{M \Omega n a_q}\right)$
    \\ \hhline{~---}
    & \xmark & $ \Theta \left( \frac{\varepsilon}{Mn}\right) $ and $\leq \frac{\varepsilon}{4nM}$ & $\mathcal{O} \left(\frac{\varepsilon^2 }{M \Omega n^2 a_q}\right)$
    \\ \hline
    
    \multirow{2}{*}{\specialcell{cube}} &  
    \cmark & $\Theta \left( \frac{\varepsilon}{M}\right)$ & $\mathcal{O} \left(\frac{\varepsilon^2 }{M \Omega n a_q}\right)$
    \\ \hhline{~---}
    & \xmark & $ \Theta \left( \frac{\varepsilon}{M\sqrt{n}}\right) $ and $\leq \frac{\varepsilon}{\sqrt{8n}M}$ & $\mathcal{O} \left(\frac{\varepsilon^2 }{M \Omega n^{3/2} a_q}\right)$
    \\ \hline
    
    \multirow{2}{*}{\specialcell{ball in \\ $p$-norm}} &  
    \cmark & $\Theta \left( \frac{\varepsilon}{M}\right)$ & $\mathcal{O} \left(\frac{\varepsilon^2 }{M \Omega n a_q}\right)$
    \\ \hhline{~---}
    & \xmark & $ \Theta \left( \frac{\varepsilon}{M\sqrt{n}}\right) $ and $\leq \frac{\varepsilon}{\sqrt{8n}M}$ & $\mathcal{O} \left(\frac{\varepsilon^2 }{M \Omega n^{3/2} a_q}\right)$
    \\ \hline
    
\end{tabular}
    \caption{$\tau$ and $\Delta$ in Corollary 1 in different cases}
    \label{summary_contr1}
\end{table}

\newpage

\section{Numerical Experiments}\label{sec:experiments}

In a series of our experiments, we compare zeroth-order Algorithm \ref{alg} ({\tt zoSPA}) proposed in this paper with Mirror-Descent algorithm from \cite{nemirovski} which uses a first-order oracle. In the main part of the paper we give only a part of the experiments, see the rest of the experiments in Section \ref{add_exp} of the Appendix.

We consider the classical saddle-point problem on a probability simplex:

\begin{eqnarray}
\label{exp_pr_4}
\min_{x\in \Delta_n}\max_{y\in \Delta_k} \left[  y^T Cx\right],
\end{eqnarray}

This problem has many different applications and interpretations, one of the main ones is a matrix game (see Part 5 in \cite{nemirovski}), i.e. the element $c_{ij}$ of the matrix are interpreted as a winning, provided that player $X$ has chosen the $i$th strategy and player $Y$ has chosen the $j$th strategy, the task of one of the players is to maximize the gain, and the opponent’s task -- to minimize.

We briefly describe how the step of algorithm should look for this case. The prox-function is $d(x) = \sum_{i=1}^n x_i \log x_i$ (entropy) and $V_x(y) = \sum_{i=1}^n x_i \log \nicefrac{x_i}{y_i}$ (KL  divergence). The result of the proximal operator is $u = \text{prox}_{z_k}(\gamma_k g(z_k, \xi^{\pm}_k,\mathbf{e}_k)) = z_k \exp(-\gamma_k g(z_k, \xi^{\pm}_k,\mathbf{e}_k))$, by this entry we mean: $u_i = [z_k]_i \exp(-\gamma_k [g(z_k, \xi^{\pm}_k,\mathbf{e}_k)]_i)$. Using the Bregman projection onto the simplex in following way $P(x) = \nicefrac{x}{\|x\|_1}$, we have
\begin{eqnarray*}
[x_{k+1}]_i = \frac{[x_k]_i \exp(-\gamma_k [g_x(z_k, \xi^{\pm}_k,\mathbf{e}_k)]_i)}{\sum\limits_{j=1}^n [x_k]_j \exp(-\gamma_k [g_x(z_k, \xi^{\pm}_k,\mathbf{e}_k)]_j)},
\end{eqnarray*}
\begin{eqnarray*}
[y_{k+1}]_i = \frac{[y_k]_i \exp(\gamma_k [g_y(z_k, \xi^{\pm}_k,\mathbf{e}_k)]_i)}{\sum\limits_{j=1}^n [y_k]_j \exp(\gamma_k [g_y(z_k, \xi^{\pm}_k,\mathbf{e}_k)]_j)},
\end{eqnarray*}
where under $g_x, g_y$ we mean parts of $g$ which are responsible for $x$ and for $y$. From theoretical results one can see that in our case, the same step must be used in Algorithm 1 and Mirror Descent from \cite{nemirovski}, because $n^{1 / q} = 1$ for $q =\infty$.

In the first part of the experiment, we take matrix $200 \times 200$. All elements of the matrix are generated from the uniform distribution from 0 to 1. Next, we select one row of the matrix and generate its elements from the uniform from 5 to 10. Finally, we take one element from this row and generate it uniformly from 1 to 5. Then we take the same matrix, but now at each iteration we add to elements of the matrix a normal noise with zero expectation and variance of 10, 20, 30, 40 \% of the value of the matrix element. The results of the experiment is on Figure \ref{fig:3}. 

\begin{figure}[h!]
\centering
\begin{minipage}{0.95\textwidth}
\includegraphics[width =  \textwidth]{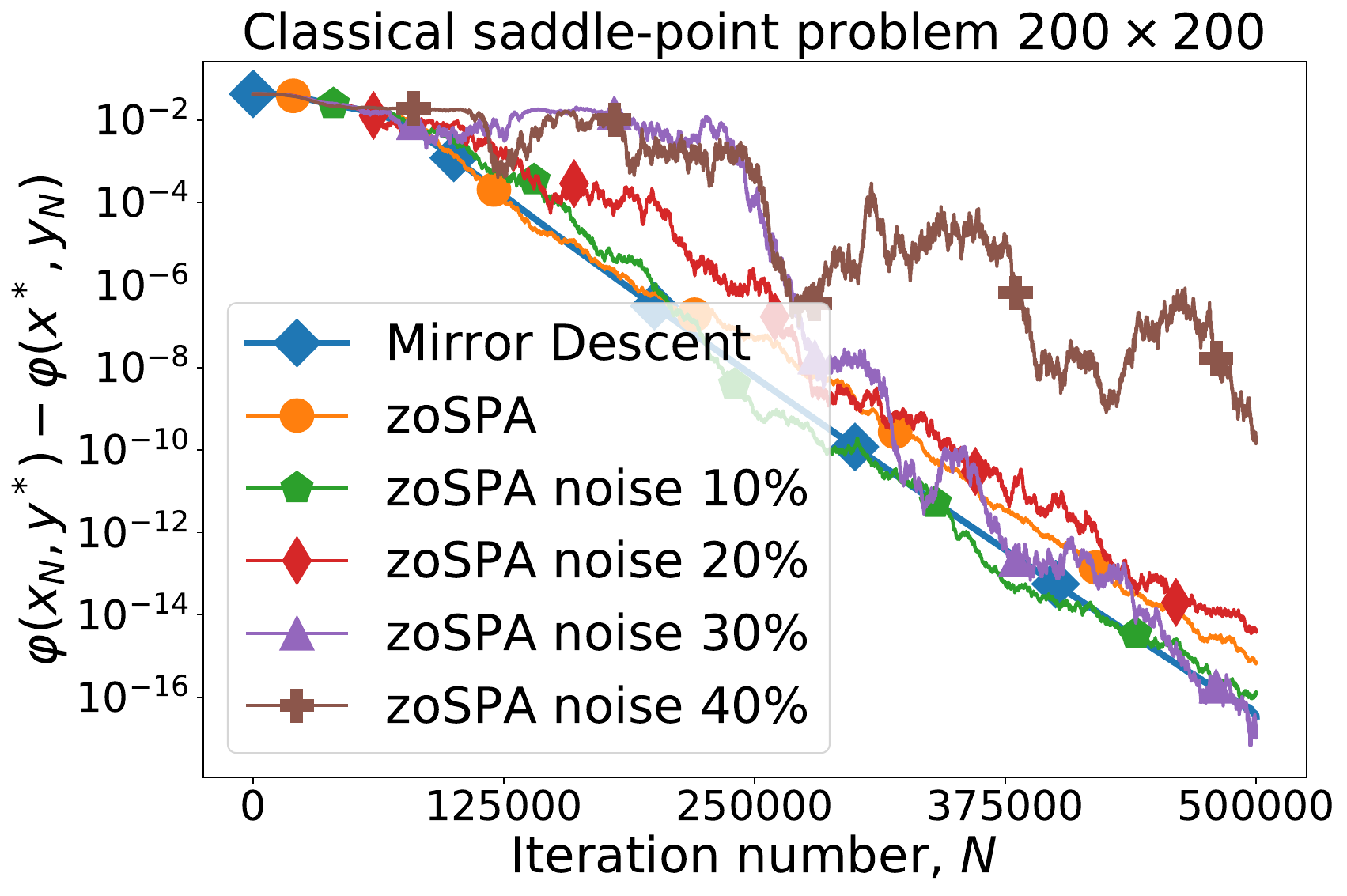}
\end{minipage}%
\caption{{\tt zoSPA} with 0 - 40 \% noise and {\tt Mirror Descent} applied to solve saddle-problem \eqref{exp_pr_4}. }
\label{fig:3}
\end{figure}

According to the results of the experiments, one can see that for the considered problems, the methods with the same step work either as described in the theory (slower $n$ times or $\log n$ times) or generally the same as the full-gradient method.

\section{Possible generalizations}\label{sec:generalizations}

In this paper, we consider non-smooth cases. Our results can be generalized for the case of strongly convex functions by using the restart technique (see for example \cite{gasnikov2017universal}). It seems that one can do it analogously.\footnote{To say in more details this can be done analogously for deterministic setup. As for stochastic setup, we need to improve the estimates in this paper by changing the Bregman diameters of the considered convex sets $\Omega$ by Bregman divergence between starting point and solution. This requires more accurate calculations (like in \cite{dvurechensky2018accelerated2}) and doesn't include in this paper. Note that all the constants, that characterized smoothness, stochasticity and strong convexity in all the estimates in this paper can be determined on the intersection of considered convex sets and Bregman balls around the solution of a radii equals to (up to logarithmic factors) the Bregman divergence between the starting point and the solution.} Generalization of the results of \cite{dvurechensky2018accelerated1,dvurechensky2018accelerated2,VorontsovaGGD19}  and \cite{alkousa2019accelerated,lin2020near} for the gradient-free saddle-point set-up is more challenging. Also, based on combinations of ideas from \cite{alkousa2019accelerated,ivanova2020oracle} it'd be interesting to develop a mixed method with a gradient oracle for $x$ (outer minimization) and a gradient-free oracle for $y$ (inner maximization). 


%
%
%
\bibliographystyle{splncs04}
\bibliography{literature}

\appendix

\section{General facts}

\begin{lemma}[see inequality 5.3.18 from \cite{nemirovski}]
Let $d(z): \mathcal{Z} \to \mathbb{R}$ is prox-function and $V_z(w)$ define Bregman divergence  associated with $d(z)$.
The following equation holds for $x,y,u \in X$:
\begin{eqnarray}
\label{temp111}
\langle \nabla d(x) - \nabla d(y), u -x \rangle = V_y(u)-V_x(u)-V_y(x).
\end{eqnarray}
\end{lemma}

\begin{lemma}[Fact 5.3.2 from \cite{nemirovski}]
    Given norm $\|\cdot \|$ on space $\mathcal{Z}$ and prox-function $d(z)$, let $z \in \mathcal{Z}$, $w \in \mathbb{R}^n$ and $z_{+} = \text{prox}_z(w)$. Then for all $u \in \mathcal{Z}$
    \begin{eqnarray}
        \label{lemma3_1}
         \langle w, z_{+} - u\rangle \leqslant V_{z}(u) - V_{z_{+}}(u) - V_{z}(z_{+}).
    \end{eqnarray}
\end{lemma}

\begin{lemma}
For arbitrary integer $n\ge 1$ and arbitrary set of positive numbers $a_1,\ldots,a_n$ we have
\begin{equation}
    \left(\sum\limits_{i=1}^m a_i\right)^2 \le m\sum\limits_{i=1}^m a_i^2.\label{eq:squared_sum}
\end{equation}
\end{lemma}

\begin{lemma}[Lemma 9 from \cite{Shamir15}]\label{lem:lemma_9_shamir} For any function $g$ which  is $L$-Lipschitz with respect to the $\ell_2$-norm, it holds that if $e$ is uniformly distributed on the Euclidean unit sphere, then 
\begin{equation*}
    \sqrt{\mathbb{E}[(g(e) - \mathbb{E}g(e))^4]} \leq c \frac{L^2}{n}
\end{equation*}
for some numerical constant $c$. One can note that $c \leq 3$.
\end{lemma}

\section{Additional experiments}\label{add_exp}

First, we present the experimental results for the classical saddle problem, which was considered in Section \ref{sec:experiments}. 

Figure \ref{fig:4} gives the results for the problem of size $200\times200$ and $500\times500$ with different noise of elements. The method for generating the matrix is the same as in Section \ref{sec:experiments}.

\begin{figure}[h]
\centering
\begin{minipage}{0.5\textwidth}
\includegraphics[width =  \textwidth]{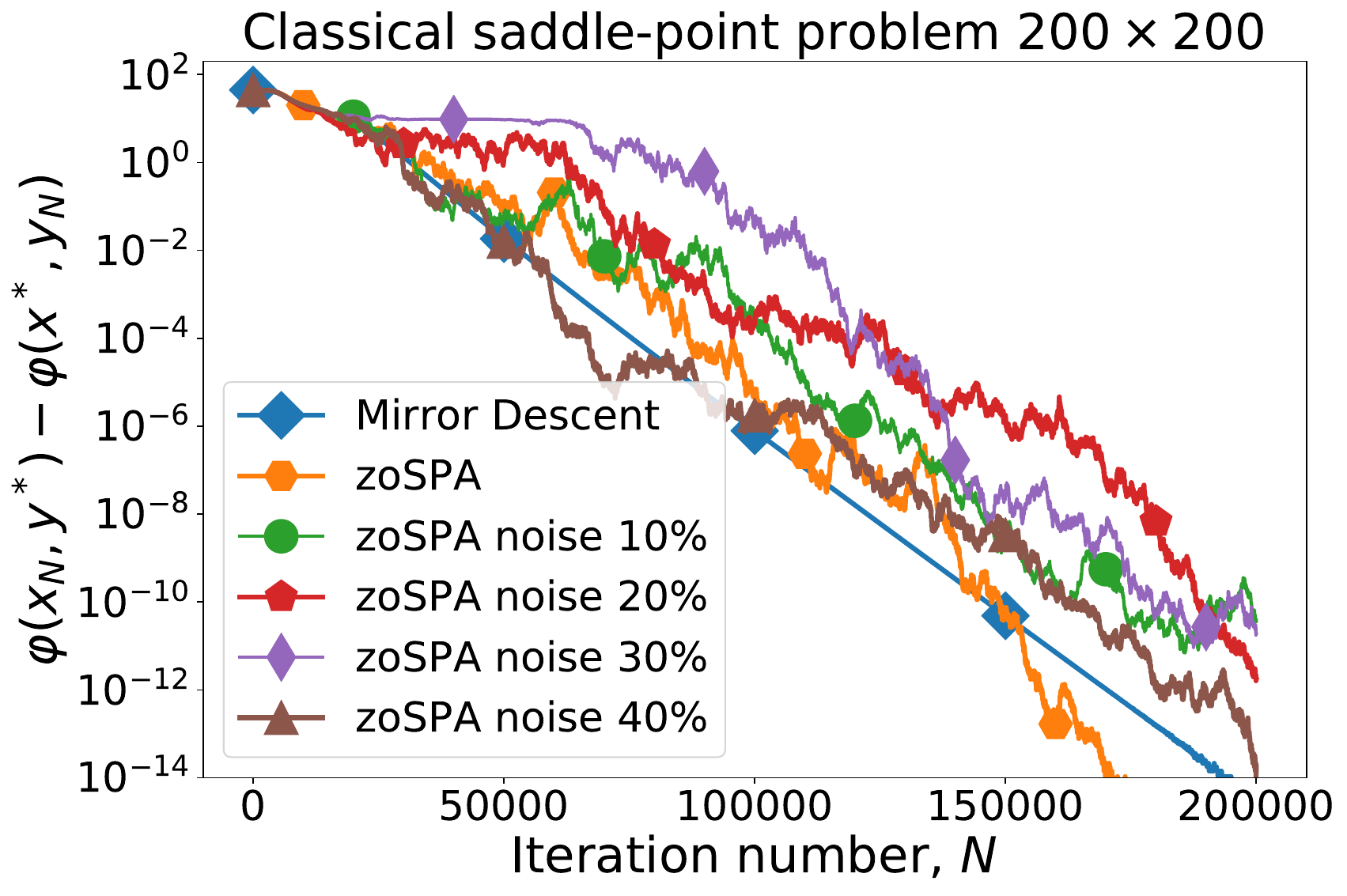}
\end{minipage}%
\begin{minipage}{0.5\textwidth}
\includegraphics[width =  \textwidth]{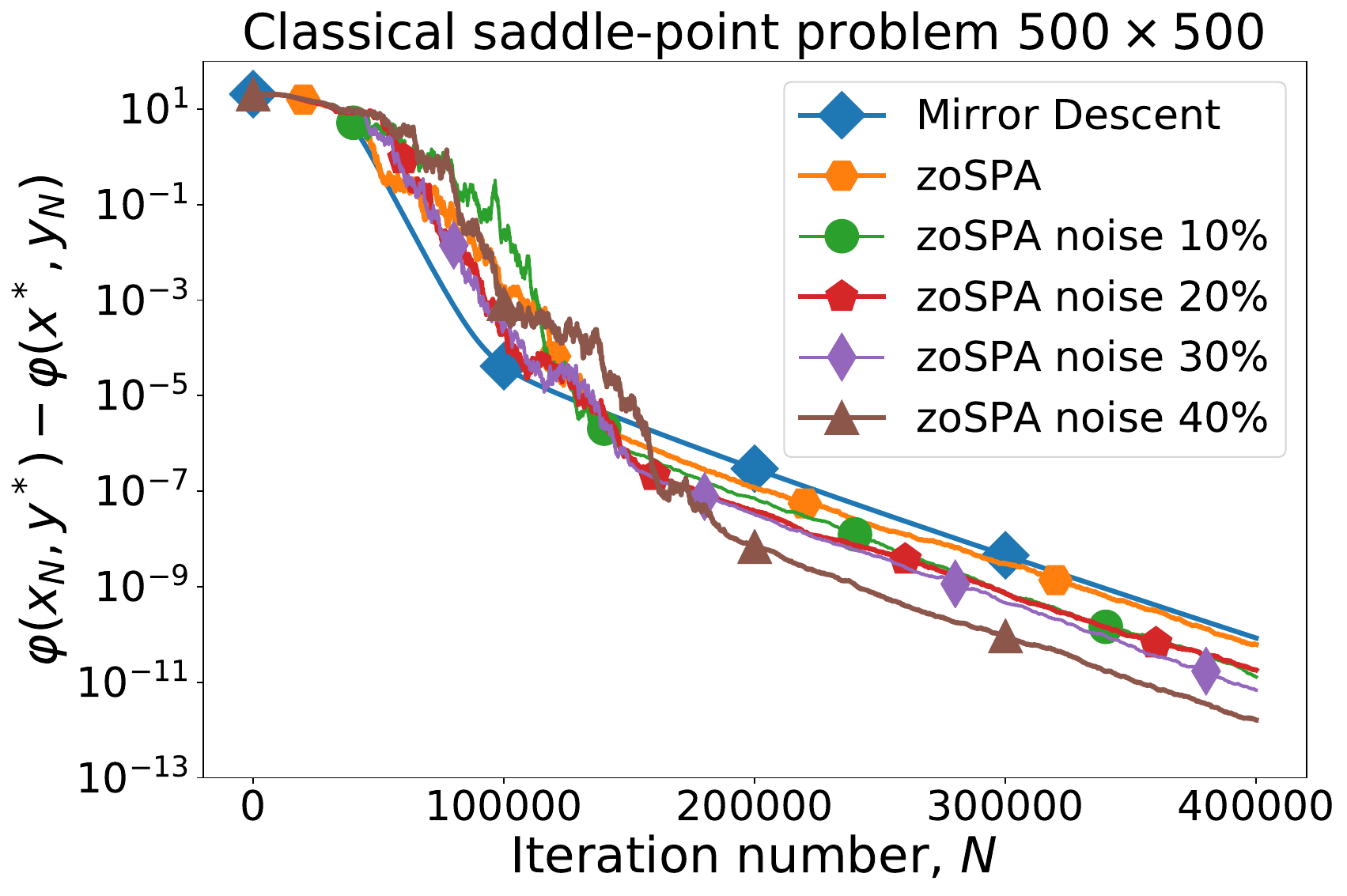}
\end{minipage}%
\\
\begin{minipage}{0.5\textwidth}
\centering
~~~(a) $200 \times 200$
\end{minipage}%
\begin{minipage}{0.5\textwidth}
\centering
~~~~~(b) $500 \times 500$
\end{minipage}%

\caption{{\tt zoSPA} with noise, {\tt Mirror Descent} applied to solve saddle-problem \eqref{exp_pr_4} size of: (a) - $200 \times 200$, (b) - $500 \times 500$.}
\label{fig:4}
\end{figure}

Next, we study how the convergence of the algorithms depends on the random generation of the matrix. We consider 3 random seeds and generate a matrix (see Section \ref{sec:experiments}). Figure \ref{fig:6} (a) shows the experimental results. One can note that the convergence rate depends on the matrix, but our Algorithm 1 and full-gradient Mirror Descent converge approximately the same for the same matrix.

Figure \ref{fig:6} (b) shows the results of an experiment where we compare the convergence of algorithms for various "saddle sizes"{}. For the first experiment we use matrix generation from Section \ref{sec:experiments}. Then we take the same matrix (do not generate it again) and multiply the row, where the saddle point is located, by 4, and we multiply the saddle point itself not by 4, but by 2. In the third experiment we do the same, but with factors of 25 and 5. And in the last case, we divide the row with the saddle-point by 2 and add 0.5 to each element. We do the following transformations, and do not generate the matrix again, for additional purity of the experiment, because using this approach, the ratio of elements in the matrix remains almost unchanged, only the "size of the saddle"{} changes. 

In the last experiment with the classical saddle problem, we change the oracle a bit: we began to take $\mathbf{e}_x$ and $\mathbf{e}_y$ with one unit and all other zeros. We conducted an experiment for a problem of $100 \times 100$.

\begin{figure}[h]
\centering
\begin{minipage}{0.33\textwidth}
\includegraphics[width =  0.95\textwidth]{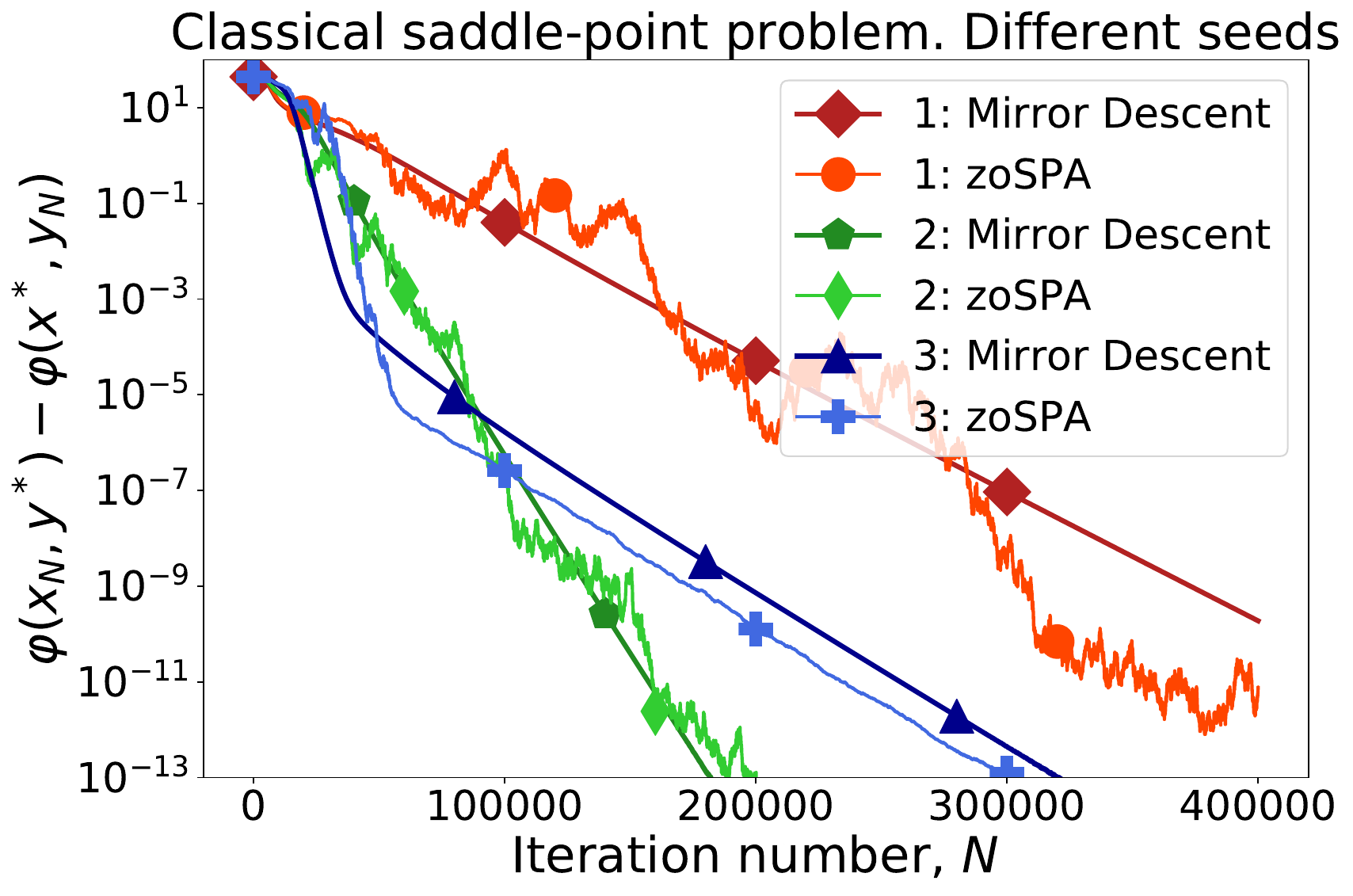}
\end{minipage}%
\begin{minipage}{0.33\textwidth}
\includegraphics[width =  \textwidth]{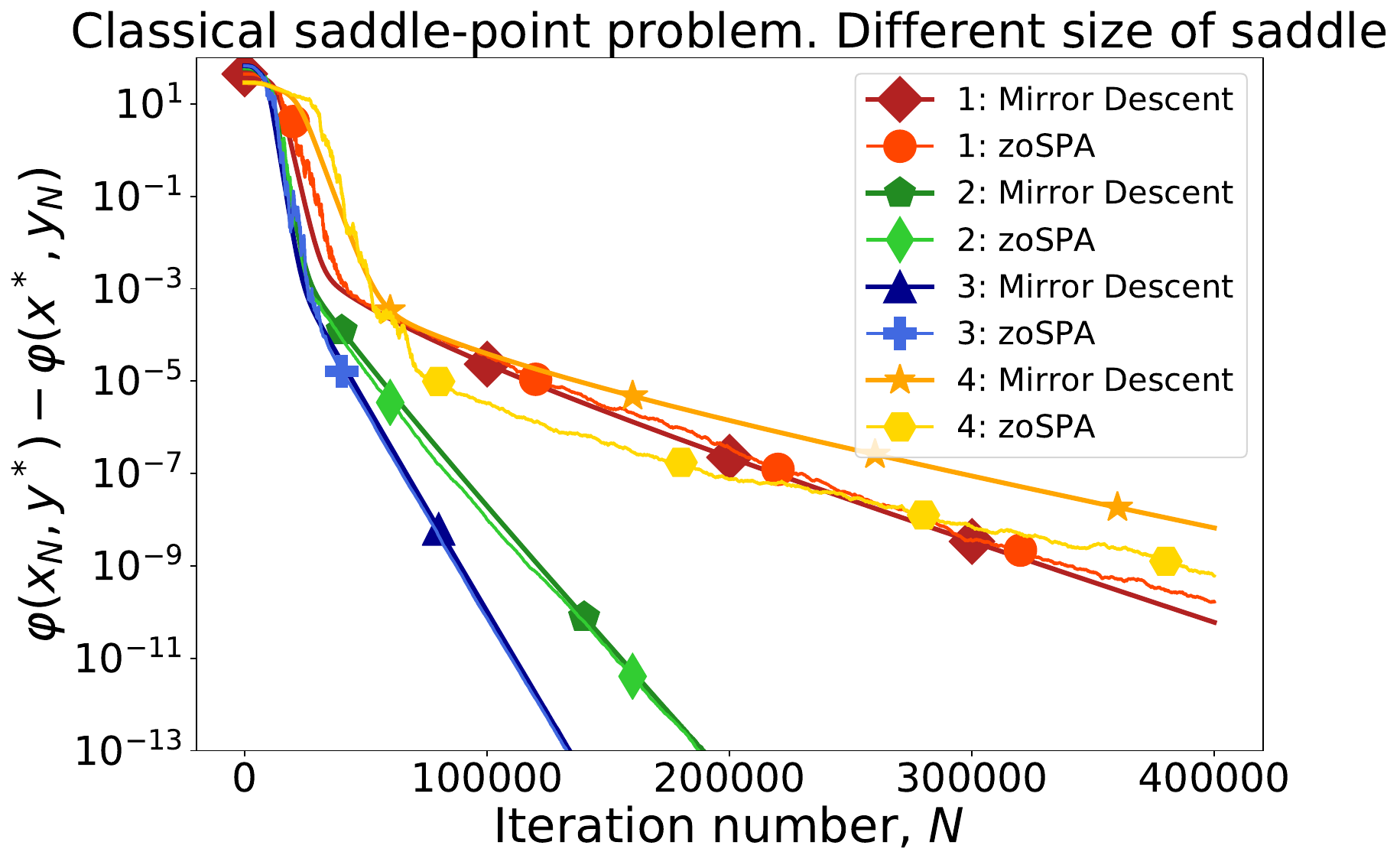}
\end{minipage}%
\begin{minipage}{0.33\textwidth}
\includegraphics[width =  \textwidth]{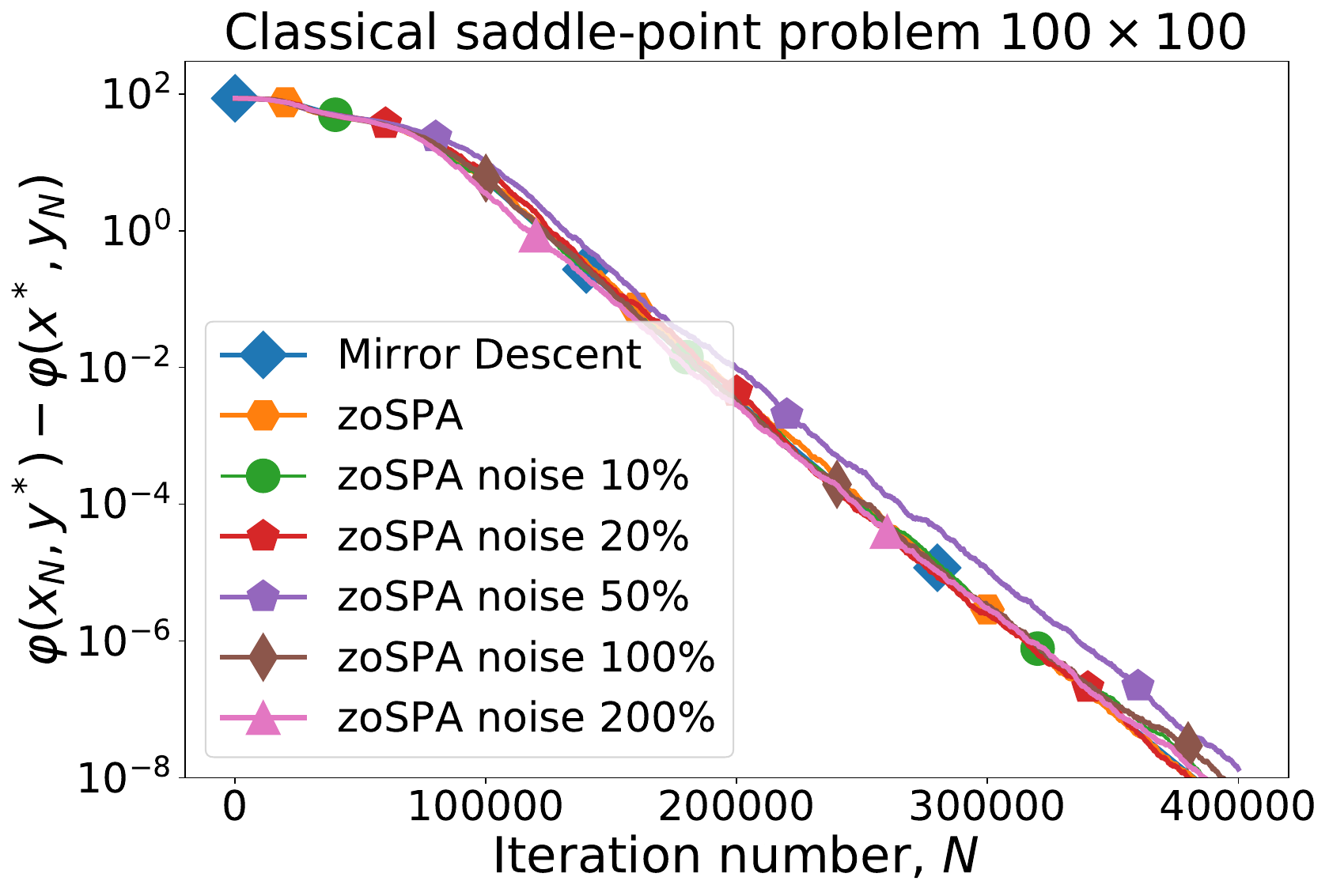}
\end{minipage}%
\\
\begin{minipage}{0.33\textwidth}
\centering
~~~(a) different seeds
\end{minipage}%
\begin{minipage}{0.33\textwidth}
\centering
~~~~~(b) different "sizes"{}
\end{minipage}%
\begin{minipage}{0.33\textwidth}
\centering
~~~~~(c) other oracle{}
\end{minipage}%

\caption{{\tt zoSPA}, {\tt Mirror Descent} applied to solve saddle-problem \eqref{exp_pr_4}: (a) - with different random seeds for matrix, (b) - with different "saddle sizes"{}, (c) - with other oracle.}
\label{fig:6}
\end{figure}

Next, we give other problems. As in the previous experiments, we compare Algorithm 1 and the full-gradient Mirror Descent. For our algorithm, we used the step is the same as in Mirror Descent and the step obtained in theory, i.e. $n^{1/q}$ times less than for Mirror Descent (see Theorem 1):

\begin{itemize}
    \item In the first experiment, we consider the monkey-saddle problem in the point $(1;1)$:
\begin{equation}
    \label{exp_pr_1}
    \min_{x\in \mathbb{R}}\max_{y\in \mathbb{R}}  \left[(x - 1)^3 - 3 (x-1)(y-1)^2 \right]
\end{equation}
\item For the second experiment, we take the problem:
\begin{equation}
    \label{exp_pr_2}
    \min_{x\in \mathbb{R}^n}\max_{y\in \mathbb{R}^n} \left[ \langle a, x - b \rangle^2 - \langle c, y - d \rangle^2 \right],
\end{equation}
where $x$, $y$ -- vectors with dimension $n = 100$, vectors $a$, $b$, $c$, $d$ are randomly generated from uniform distribution on $[0,1]$. 

\item In the third and fourth experiments we consider the following problem:
\begin{eqnarray}
\min_{x \in \mathbb{R}^n} & &\frac{1}{2} x^T A x - b^T x \nonumber\\
\text{s.t.} & &Cx = d,
\label{temp_exp}
\end{eqnarray}
where $A \in \mathbb{S}^{n}$ (symmetric matrix), $C$ is a matrix size of $k \times n$. One can rewrite \eqref{temp_exp} in the following way:
\begin{eqnarray}
\label{exp_pr_3}
\min_{x\in \mathbb{R}^n}\max_{y\in \mathbb{R}^k} \left[ \frac{1}{2} x^T A x - b^T x + y^T Cx - y^T d\right],
\end{eqnarray}
where the expression in brackets is the Lagrange function $L(x,y)$, and $y$ -- Lagrange multiplier. Problem \eqref{exp_pr_3} is a saddle-point problem.

In the third experiment we take $n=100$, $k = 10$. Matrix $A$ is positive finite. We first randomly generated positive eigenvalues, then converted them into a diagonal matrix $D$, then we made an orthogonal matrix $Q$ from random squared matrix using QR-decomposition, and finally we get $A$ by $A = Q^T D Q$. Matrix $C$ and vectors $b$, $d$ are obtained randomly.
\begin{figure}[H]
\centering
\begin{minipage}{0.33\textwidth}
\includegraphics[width =  \textwidth]{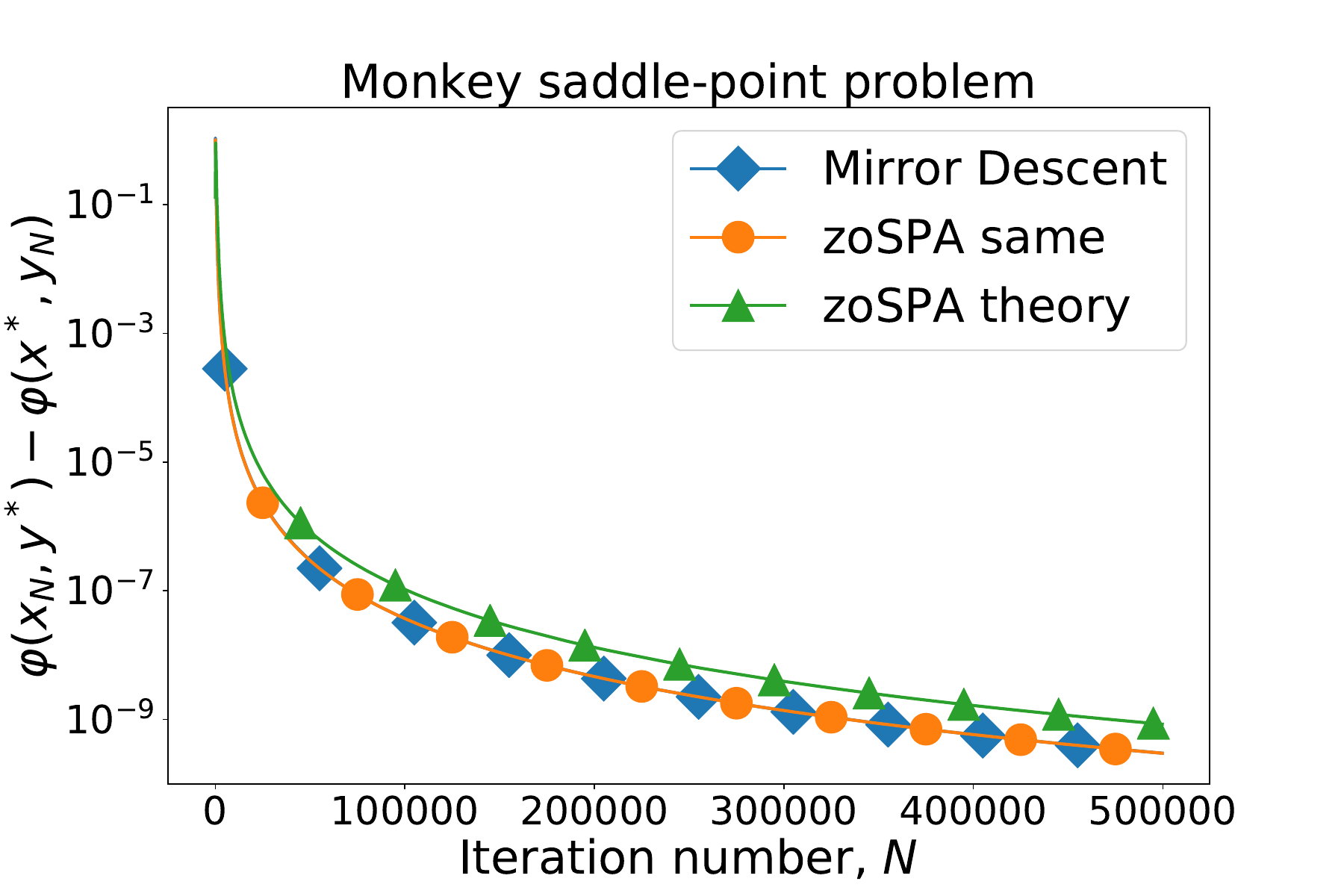}
\end{minipage}%
\begin{minipage}{0.33\textwidth}
  \centering
\includegraphics[width =  \textwidth]{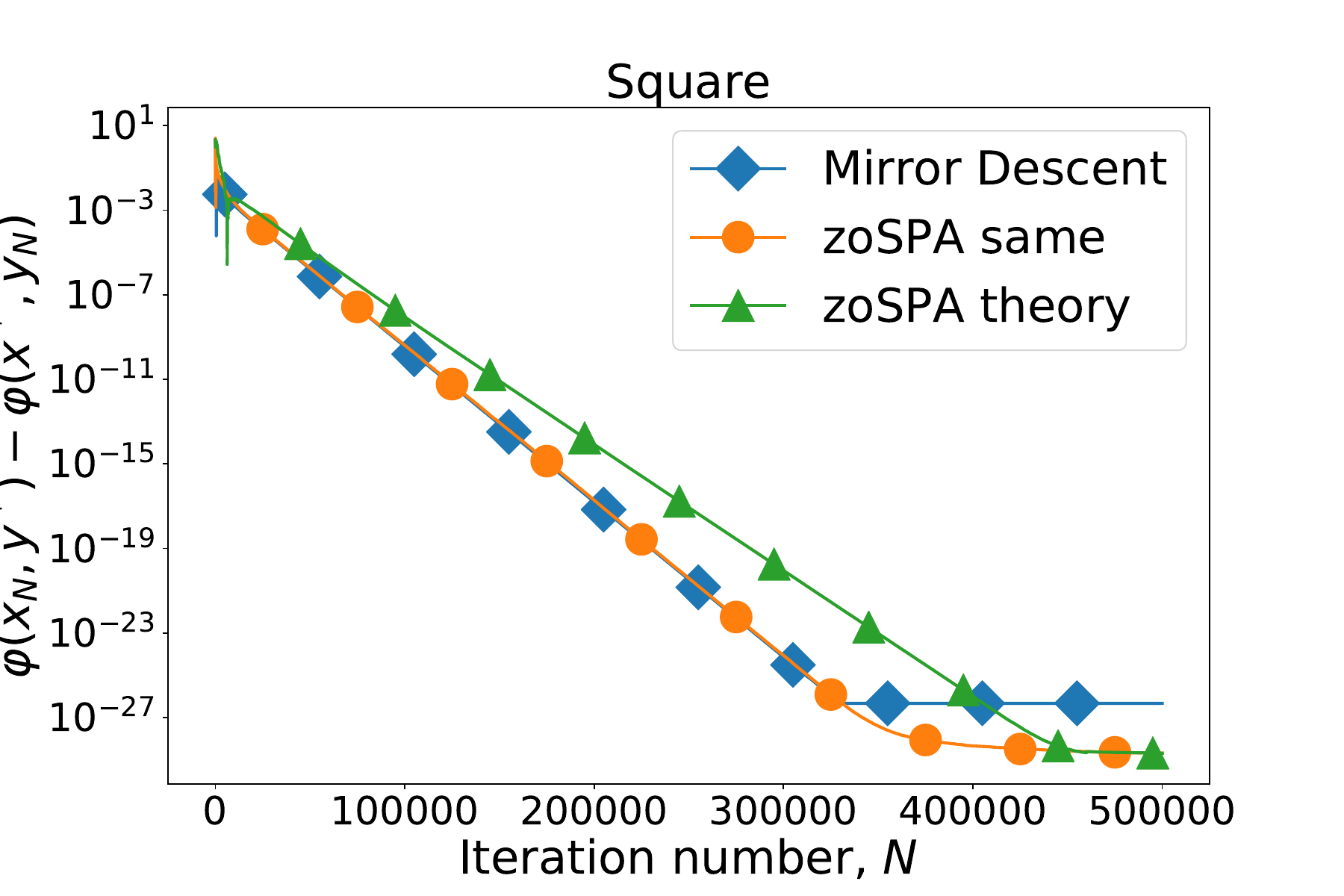}
\end{minipage}%
\begin{minipage}
{0.33\textwidth} \includegraphics[width =  \textwidth]{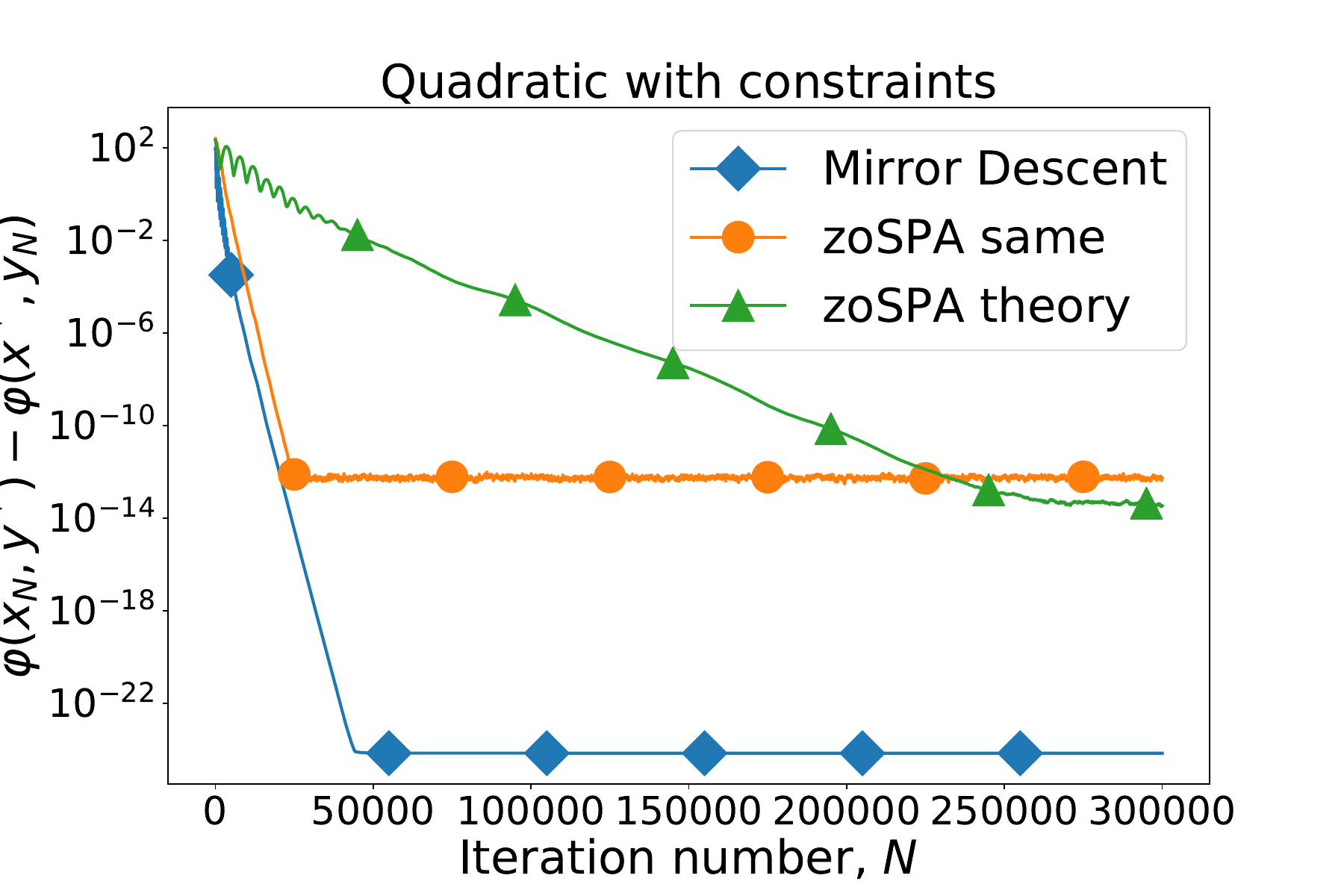}
\end{minipage}%
\\
\begin{minipage}{0.33\textwidth}
\centering
~~~(a) 
\end{minipage}%
\begin{minipage}{0.33\textwidth}
\centering
~~~~(b) 
\end{minipage}%
\begin{minipage}{0.33\textwidth}
\centering
~~~~(c) 
\end{minipage}%
\caption{{\tt zoSPA}, {\tt Mirror Descent} applied to solve saddle-problem: (a)\eqref{exp_pr_1}, (b) \eqref{exp_pr_2}, (c)\eqref{exp_pr_3}.}
\label{fig:1}
\end{figure}

\item In the fourth experiment we take $n=100$, $k = 100$. Matrix $A$ is positive semi-definite. In the first case, we start at zero point (as described in the algorithm). In the second case, we start at a point that is close to the solution.

Figure \ref{fig:2} shows the experimental results: (a) real convergence from the zero point, (b) scaled convergence, i.e. two convergence graphs are taken starting from 200,000 iterations and are scaled so that the lines come from one point, (c) the real convergence from a point close to the solution.
\begin{figure}[h]
\centering
\begin{minipage}{0.33\textwidth}
\includegraphics[width =  \textwidth]{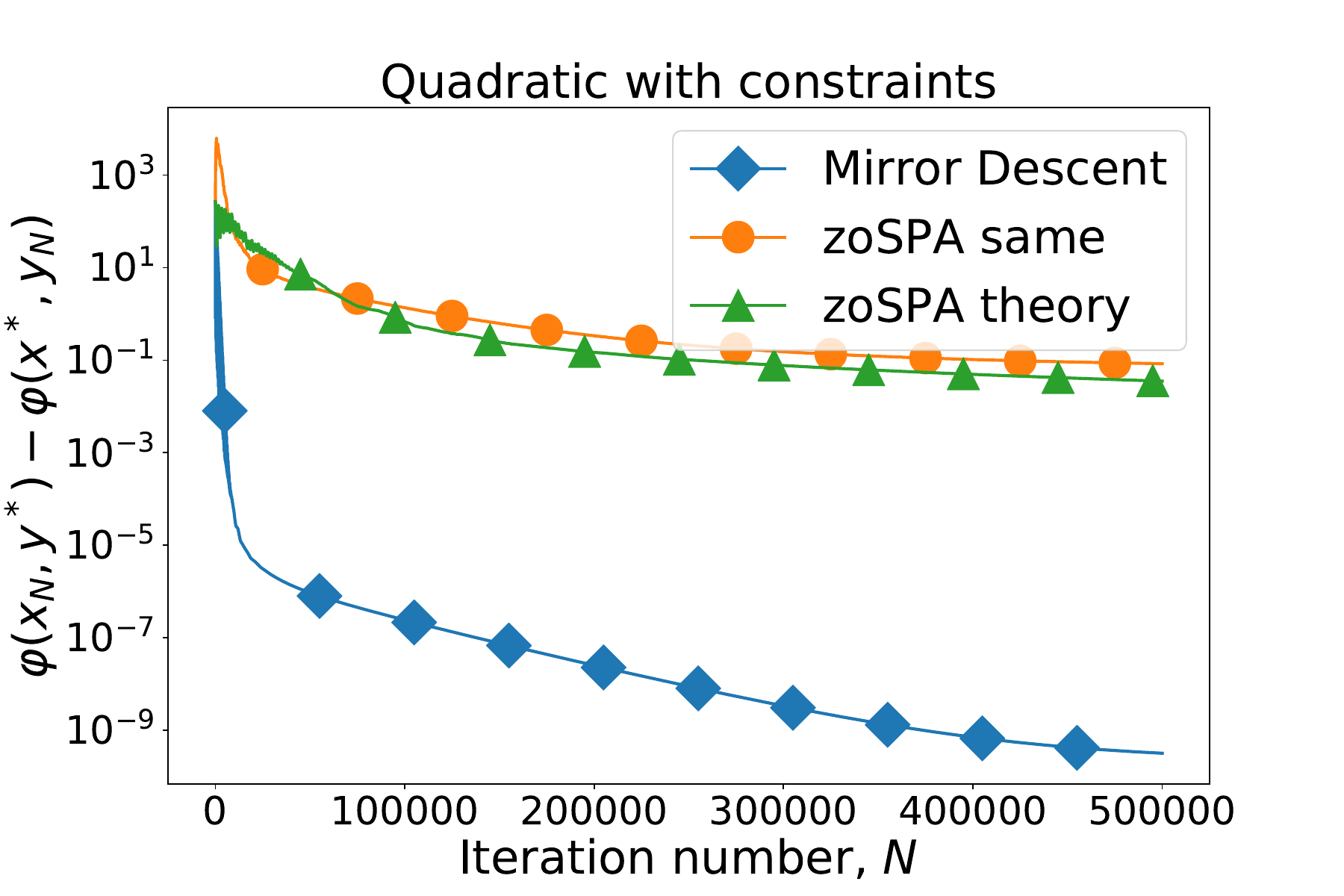}
\end{minipage}%
\begin{minipage}
{0.33\textwidth} \includegraphics[width =  \textwidth]{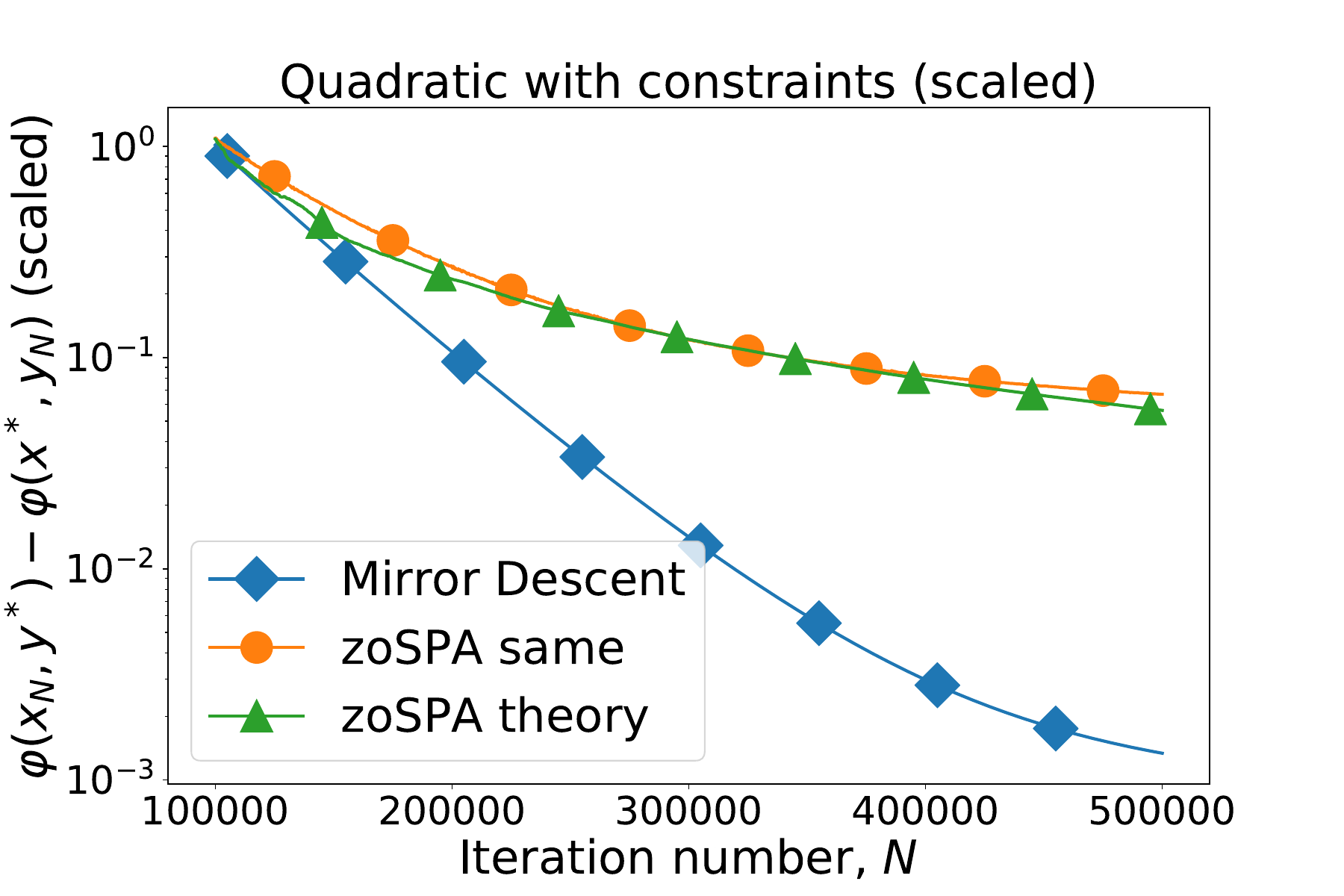}
\end{minipage}%
\begin{minipage}
{0.33\textwidth} \includegraphics[width =  \textwidth]{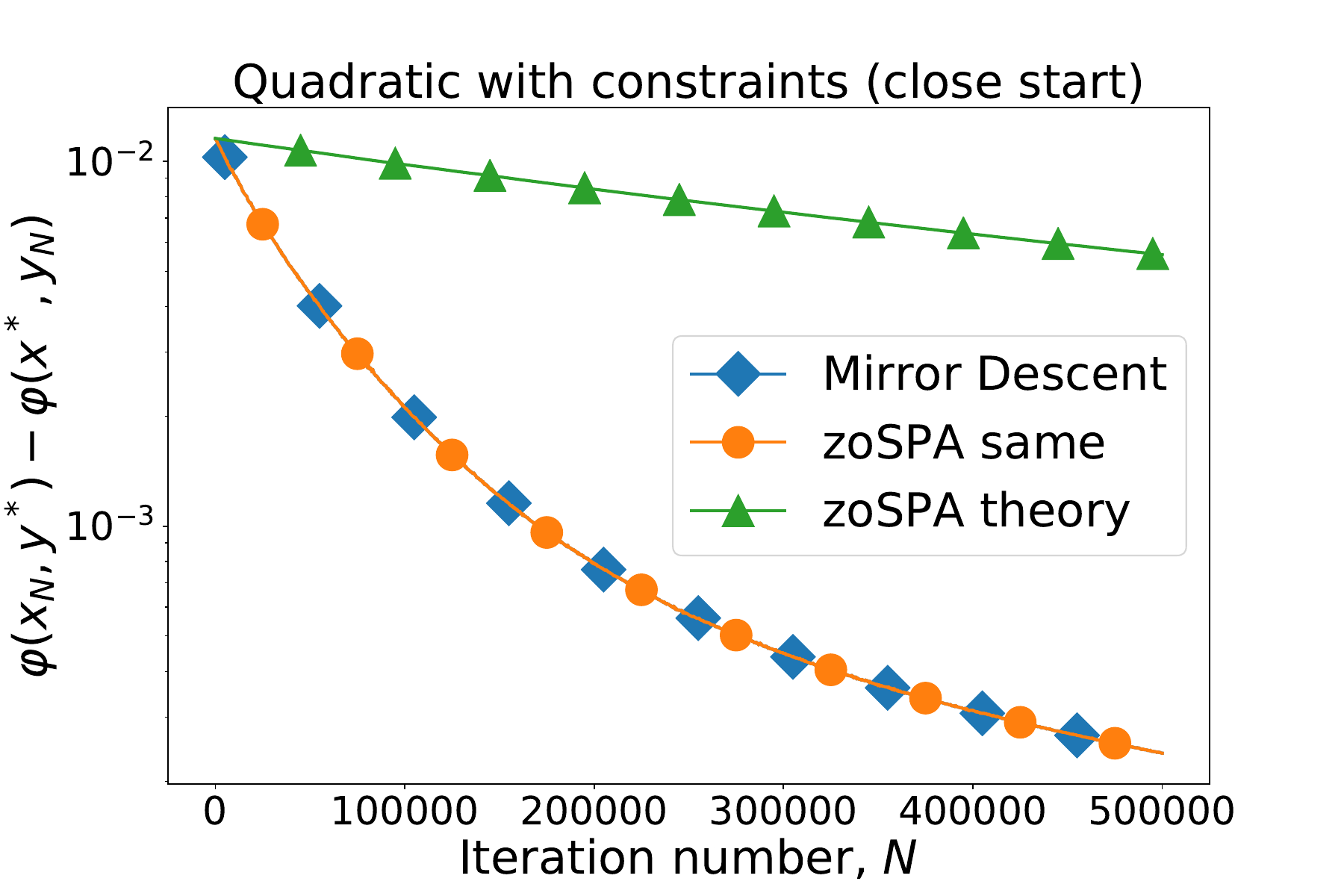}
\end{minipage}%
\\
\begin{minipage}{0.33\textwidth}
\centering
~~~(a) 
\end{minipage}%
\begin{minipage}{0.33\textwidth}
\centering
~~~~(b) 
\end{minipage}%
\begin{minipage}{0.33\textwidth}
\centering
~~~~(c) 
\end{minipage}%
\caption{{\tt zoSPA}, {\tt Mirror Descent} applied to solve saddle-problem \eqref{exp_pr_3}: (a) zero start, (b) zero start (scaled), (c) close start.}
\label{fig:2}
\end{figure}

\end{itemize}

\end{document}